\documentclass[11pt,reqno]{amsart}

\usepackage[english]{babel}
\usepackage[utf8]{inputenc}
\usepackage{latexsym}
\usepackage{verbatim}
\usepackage{rotating}
\usepackage[colorlinks=true,citecolor=black,linkcolor=black,urlcolor=blue]{hyperref}
\usepackage{amssymb,amsfonts,amsmath,stmaryrd}

\usepackage{color}
\usepackage{listings}
\usepackage{caption}
\usepackage{amsaddr}

\usepackage{enumitem}
\setlistdepth{5}

\newlist{myitemize}{itemize}{5}
\setlist[myitemize,1]{label=a}
\setlist[myitemize,2]{label=b}
\setlist[myitemize,3]{label=c}
\setlist[myitemize,4]{label=d}
\setlist[myitemize,5]{label=e}

\newtheorem{Theorem}{Theorem}
\newtheorem{Lemma}[Theorem]{Lemma}
\newtheorem{Proposition}[Theorem]{Proposition}
\newtheorem{Definition}[Theorem]{Definition}
\newtheorem{Corollary}[Theorem]{Corollary}

\newtheorem{Example}[Theorem]{Example}
\newtheorem{Remark}[Theorem]{Remark}

\newcommand{\beq}{\begin{equation}}
\newcommand{\eeq}{\end{equation}}
\def\emm#1,{{\em #1}}

\newcommand{\dba}[1]{\mathop{\overrightarrow{\overleftarrow#1}}}

\def \Z {\mathbb Z}
\def \N {\mathbb N}  
\def \R {\mathbb R}  

\def \C {\mathcal C}  
\def \T {\mathcal T}

\def \Sb {\mathcal B}  
\def \Sf {\mathcal F}  
\def \Ss {\mathcal S}  

\def \origin {\mathcal O}

\def \D {\Delta}

\def \bsa {\overline{s_1}}
\def \bsb {\overline{s_2}}
\def \bsc {\overline{s_3}}

\def \lra {\longleftrightarrow}

\def \Mo#1#2#3{m_{#2}(#3)}
\def \Mop#1#2#3{m'_{#2}(#3)}

\newcommand{\f}{g}

\def \D {\Delta}

\newcounter{nalg} 
\DeclareCaptionLabelFormat{algocaption}{Algorithm \thenalg} 

\lstnewenvironment{algorithm}[1][] 
{   
    \refstepcounter{nalg} 
    \captionsetup{labelformat=algocaption,labelsep=colon} 
    \lstset{ 
        mathescape=true,
        frame=tB,
        keywordstyle=\color{black}\bfseries\em,
        keywords={,input, output, return, datatype, function, in, if, else, foreach, while, begin,  for, from, to, do,make, on, between, and, decreasing, metadata, add,} 
        xleftmargin=.04\textwidth,
        #1 
    }
}
{}

\catcode`\@=11
\def\section{\@startsection{section}{1}%
 \z@{.7\linespacing\@plus\linespacing}{.5\linespacing}%
 {\normalfont\bfseries\scshape\centering}}

\def\subsection{\@startsection{subsection}{2}%
  \z@{.5\linespacing\@plus\linespacing}{.5\linespacing}%
  {\normalfont\bfseries\scshape}}

\def\subsubsection{\@startsection{subsubsection}{3}%
 \z@{.5\linespacing\@plus\linespacing}{-.5em}
 {\normalfont\bfseries}}
\catcode`\@=12

\graphicspath{{images/}}

\begin{document}
\title[Bijections between triangular walks and Motzkin paths]
{Bijections between 
walks
inside a triangular domain
and 
Motzkin paths of bounded amplitude}

\author{Julien Courtiel}
\address{ \vspace*{-.5cm} \small Normandie University, UNICAEN, ENSICAEN, CNRS, GREYC}
\author{Andrew Elvey Price}
\address{ \vspace*{-.5cm}  Universit\'e de Bordeaux, LaBRI, Universit\'e de Tours, IDP }
\author{Ir\`ene Marcovici}
\address{ \vspace*{-.5cm}  \small Universit\'e de Lorraine, CNRS, Inria, IECL, F-54000 Nancy, France}

\begin{abstract}
This paper solves an open question of Mortimer and Prellberg asking for an explicit bijection between two families of walks. The first family is formed by what we name \textit{triangular walks}, which are two-dimensional walks moving in six directions ($0^{\circ}$, $60^{\circ}$, $120^{\circ}$, $180^{\circ}$, $240^{\circ}$, $300^{\circ}$) and confined within a triangle. The other family is comprised of two-colored Motzkin paths with bounded height, in which the horizontal steps may be forbidden at maximal height.

We provide several new bijections. The first one is derived from a simple inductive proof, taking advantage of a $2^n$-to-one function from generic triangular walks to triangular walks only using directions $0^{\circ}$, $120^{\circ}$, $240^{\circ}$. The second is based on an extension of Mortimer and Prellberg's results to triangular walks starting not only at a corner of the triangle, but at any point inside it. It has a linear-time complexity and is in fact adjustable: by changing some set of parameters called a \textit{scaffolding}, we obtain a wide range of different bijections. 

Finally, we extend our results to higher dimensions. In particular, by adapting the previous proofs, we discover an unexpected bijection between three-dimensional walks in a pyramid and two-dimensional simple walks confined in a bounded domain shaped like a waffle. 
\end{abstract}

\maketitle

\subsection*{Thanks}
\thanks{JC was supported by the ``\textit{CNRS projet JCJC}" named ASTEC. AEP was supported by the European Research Council (ERC) in the European Union’s Horizon 2020 research and innovation programme, under the Grant Agreement No.~759702.  The authors want also to thank the sponsors of the conference ALEA Young (ANR-MOST MetAConC, Normastic, Université de Caen Normandie) without which this collaboration would never have been born.}

\section{Introduction}

In part due to the ubiquity of random walks in probability theory, lattice walks are extensively studied in enumerative combinatorics~\cite{Mohanty,Humphreys,BoMi10}. In this context, it is frequently discovered that two families of walks, which seem to be very different, are in fact counted by the same numbers. The initial proof is often not combinatorial, 
and finding an explicit bijection between such families can prove to be a difficult task (see for example \cite{Eliz15,basket}). 

In this spirit, this paper answers a $5$ year old open question from Mortimer and Prellberg~\cite[Section 4.3]{MortimerPrellberg}. By solving a functional equation satisfied by the generating function, the two authors realized that the number of walks in a triangular domain starting from a corner of this domain is equal to the number of Motkzin paths of bounded height -- we will give precise definitions of these families in the following subsections.  
Their proof was purely analytic and,  
consequently, it raised the issue of finding an explanatory bijection. This gave rise to an open question, which became rather famous in the community, since Prellberg, one of the authors of \cite{MortimerPrellberg}, regularly asked for a bijection in open problems sessions during combinatorics conferences. The current paper solves this question, in several manners. 

In the rest of this section, we introduce the notions of triangular paths, Motzkin paths and Motzkin meanders, which will be our objects of study, and we present more formally Mortimer and Prellberg's problem. Then, in the last subsection, we give a detailed outline of the present paper.

\subsection{Triangular paths}

Let $(e_1,e_2,e_3)$ denote the standard basis of $\R^3$. For some $L\in\N$, we define the subset $\T_L$ of $\N^3$ as the triangular section of side length $L$ of the integer lattice:
$$\T_L=\{x_1\,e_1+x_2\,e_2+x_3\,e_3 : x_1, x_2, x_3\in\N, x_1+x_2+x_3=L\}.$$
An example of such lattice is shown by Figure~\ref{figure:T3} (left).

We also introduce the notation
$$s_1=e_1-e_3, \quad s_2=e_2-e_1, \quad s_3=e_3-e_2,$$
and for $i\in\{1,2,3\}$, we set $\overline{s_i}=-s_i.$ 
We will interpret the vectors $s_i$ as \emph{forward steps} and the vectors $\overline{s_i}$ as \emph{backward} steps. 
We denote by $\Sf=\{s_1,s_2,s_3\}$ and $\Sb=\{\bsa,\bsb,\bsc\}$ the set of forward and backward steps, respectively.

\begin{figure}\begin{minipage}{0.4\textwidth}
\begin{center}
\includegraphics[scale=1.2]{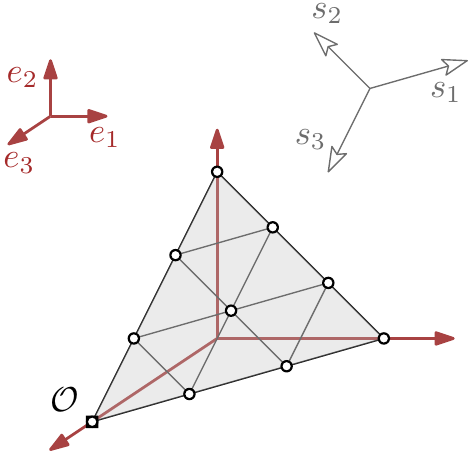}
\end{center}
\end{minipage}
\hspace{2cm}
\begin{minipage}{0.4\textwidth}
\begin{center}
\includegraphics[scale=1.2]{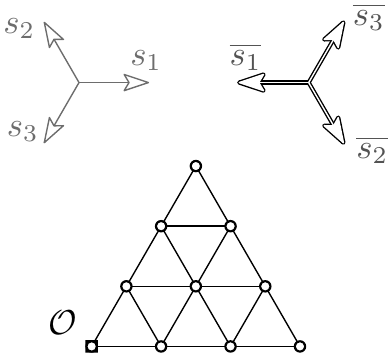}
\end{center}
\end{minipage}

\caption{\textit{Left.} 
The triangular lattice $\T_3$. \textit{Right.} The planar representation of the same lattice, with  $\Sf$ and $\Sb$.
}
\label{figure:T3}
\end{figure}

For convenience, we define the indices modulo $3$, thus $s_0 = s_3$ and $s_4 = s_1$.

The triangular lattice $\T_L$ can be naturally drawn in the plane, as an equilateral triangle of side length $L$, subdivided in smaller equilateral triangles of side length $1$ (see Figure~\ref{figure:T3} right). We will use this planar representation for the remainder of the document.

We define $\origin$ as the bottom left corner of $\T_L$, that is to say $\origin = L e_3$. In some sense, it denotes an origin for the lattice $\T_L$.

\begin{Definition}[Forward paths, triangular paths]
Given an integer $L\in\N$, and a point $z\in \T_L$, a \emph{forward (triangular) path} of length $n$ starting from $z$ is a sequence $(\sigma_1,\ldots,\sigma_n)\in \Sf^{n}$ satisfying
\[\forall k\in\{0,\ldots,n\}, \quad z+\sum_{i=1}^k \sigma_i \in \T_L.\]
A \emph{(generic) (triangular) path} of length $n$ starting from $z$ is a sequence $(\omega_1,\ldots,\omega_n) \in \left(\Sf \cup \Sb\right)^n$ satisfying
\[\forall k\in\{0,\ldots,n\}, \quad z+\sum_{i=1}^k \omega_i \in \T_L.\] 
\end{Definition}

\begin{figure}
\begin{center}
\includegraphics[width = 0.9 \textwidth]{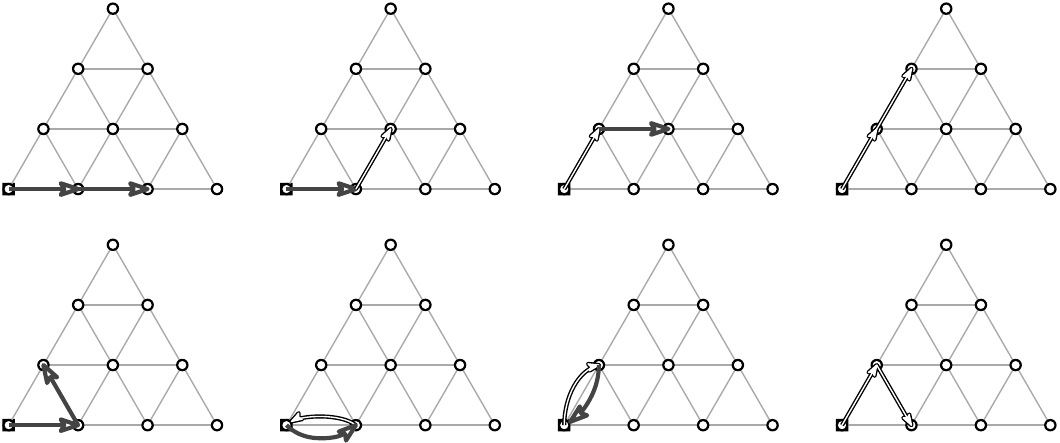}
\end{center}
\caption{All triangular paths of $\T_3$ with length $2$ starting at $\origin$.}
\label{figure:triangular}
\end{figure}

If $L \geq 2$, there are $2$ forward paths of length $2$ and $8$ generic paths of length $2$ starting from $\origin$, as shown by Figure~\ref{figure:triangular}.

For those who are familiar with the enumeration of walks in the quarter of plane, forward paths can be seen as a subfamily of 
tandem walks~\cite[Section 4.7]{yellowBook}. \textit{Tandem walks} are walks on $\N^2$ using steps $(1,0)$, $(-1,1)$, $(0,-1)$ (East, North-West, South steps). 
Their name comes from the fact that in queuing theory, they model the behavior of two queues in series.

To be precise, forward paths of $\T_L$ are
equivalent to tandem walks confined in the part of the positive quarter plane below the anti-diagonal $x+y = L$. In terms of queues, forward paths can be represented by two queues in series where the total number of jobs (or customers) in both queues is never greater than $L$.

\begin{figure}
\begin{center}
\includegraphics[scale = 1.3]{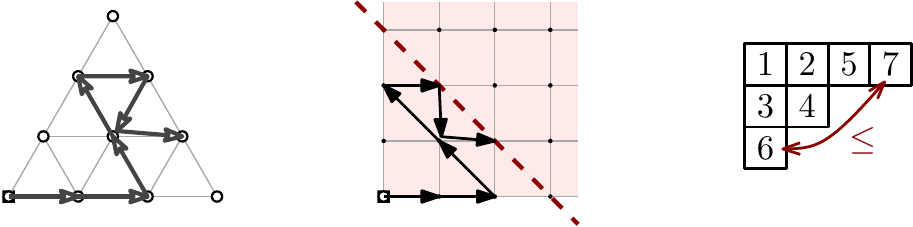}
\end{center}
\caption{Equivalent definitions of the same object: forward paths of $\T_3$ (left); tandem walks in the positive quarter of plane and below the antidiagonal $x + y = 3$ (middle); standard Young tableaux with three rows or less such that the label of the $i$th cell of the bottom row must be less than the label of $(i+3)$th cell of the top row (right).  }
\label{figure:tandem}
\end{figure}

Since tandem walks are also described by \textit{standard Young tableaux}~\cite{young} with three rows or less, forward paths on $\T_L$ form a particular subfamily of standard Young tableaux: they must have $3$ rows or less, and for every $k > L$, if there is a $k$th cell in the top row of the tableau, then its label must be greater than the label of the $(k-L)$th cell of the third row (which must exist). The three equivalent definitions of forward paths are illustrated by Figure~\ref{figure:tandem}.

As for generic triangular paths, they are naturally encoded by \textit{double-tandem walks}, which are walks on $\N^2$ using steps $(1,0)$, $(-1,1)$, $(0,-1)$, $(-1,0)$,  $(1-,1)$, $(0,1)$ (we add to the base step set of the tandem walks the opposite steps).

\subsection{Motzkin paths and meanders}
\label{ss:motzkin}

A \textit{Motzkin path} is a path using up, horizontal and down steps, respectively denoted $\nearrow$, $\rightarrow$ and $\searrow$, such that:
\begin{itemize}
\item it starts at height $0$;
\item it remains at height $\geq 0$ (i.e. inside any prefix of a Motzkin path, the number of $\nearrow$ steps is greater or equal to the number of $\searrow$ steps);
\item it ends at height $0$ (i.e. in total, there are as many $\nearrow$ steps as $\searrow$ steps).
\end{itemize}

The following definition refines the notion of maximum height for a Motzkin path.

\begin{Definition}[Amplitude]
Let $M$ be a Motzkin path and $H$ its maximum height (i.e the maximal difference between the number of $\nearrow$ steps and the number of $\searrow$ steps in a prefix of $M$).

The \emph{amplitude} of $M$ is defined as
\[ \left\{ \begin{array}{cl}
 2H +1 &  \textrm{if a horizontal step }\rightarrow\textrm{ is performed at height }H,
\\ 2 H & \textrm{otherwise.}
\end{array}  \right.\]
\end{Definition}

\begin{figure}
\begin{center}
\includegraphics[width = 0.95 \textwidth]{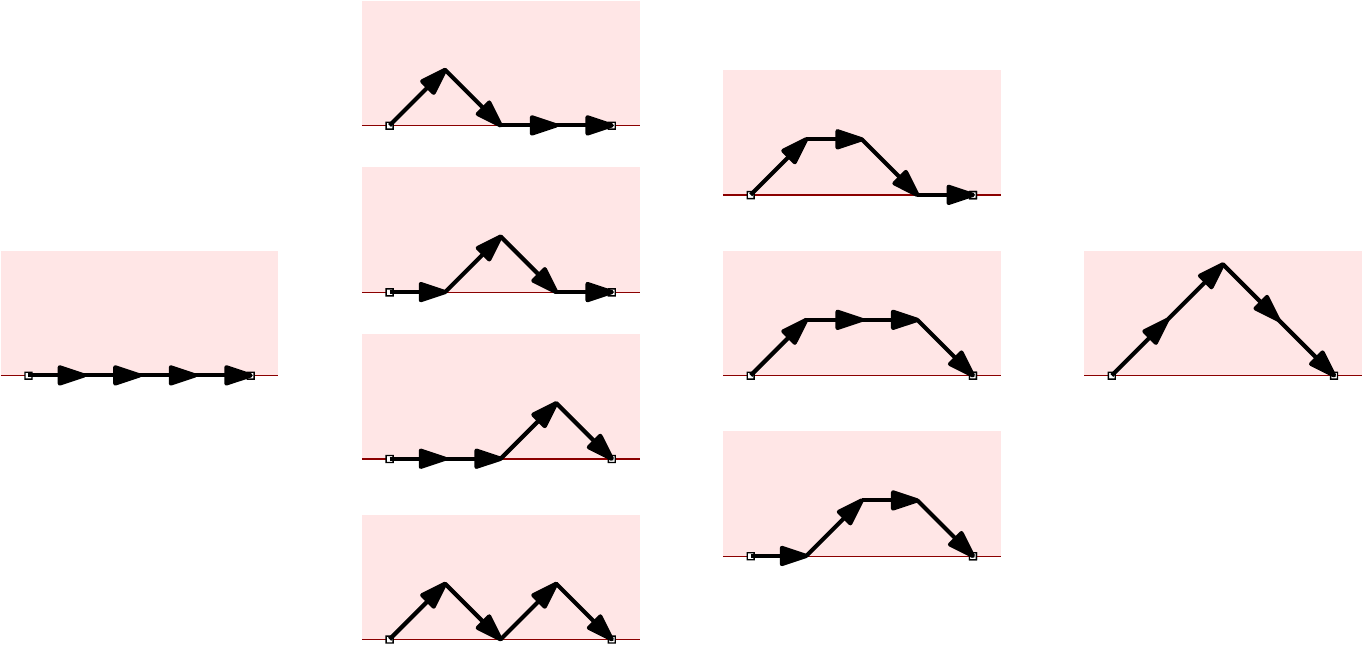}
\end{center}
\caption{Motzkin paths of length $4$ sorted with respect to their amplitude (from $1$ to $4$)}
\label{figure:motzkin}
\end{figure}

For example, all the Motzkin paths of length $4$ are listed by Figure~\ref{figure:motzkin}: there is one such path with amplitude $1$, four with amplitude $2$, three with amplitude $3$ and one with amplitude $4$.

A \textit{Motzkin meander} is a suffix\footnote{Usually a meander is defined as a prefix, but up to a vertical symmetry, it is equivalent.} of a Motzkin path. A Motzkin meander can thus start at any height, but must end at height $0$.

\subsection{Mortimer and Prellberg's open question}

We now state Mortimer and Prellberg's enumerative result (reformulated in terms of amplitude), for which we are going to give explanatory bijections. 

\begin{Theorem}[Corollary 4 \cite{MortimerPrellberg}]
Given any $L \geq 0$, there are as many triangular paths in $\T_L$ starting at $\origin$ with $p$ forward steps and $q$ backward steps as bicolored Motzkin paths of length $p+q$ with an amplitude less than or equal to $L$ where $p$ steps are colored in black and $q$ are colored in white.
\label{theo:mortimerprellberg}
\end{Theorem}

Setting $p=n$ and $q=0$, we obtain the following corollary about forward paths.

\begin{Corollary}
Given any $L \geq 0$, there are as many forward paths in $\T_L$ of length $n$ starting at $\origin$ as Motzkin paths of length $n$ with an amplitude less than or equal to $L$.
\label{cor:forward-motzkin}
\end{Corollary}

\begin{figure}
\begin{center}
\includegraphics[width=0.9\textwidth]{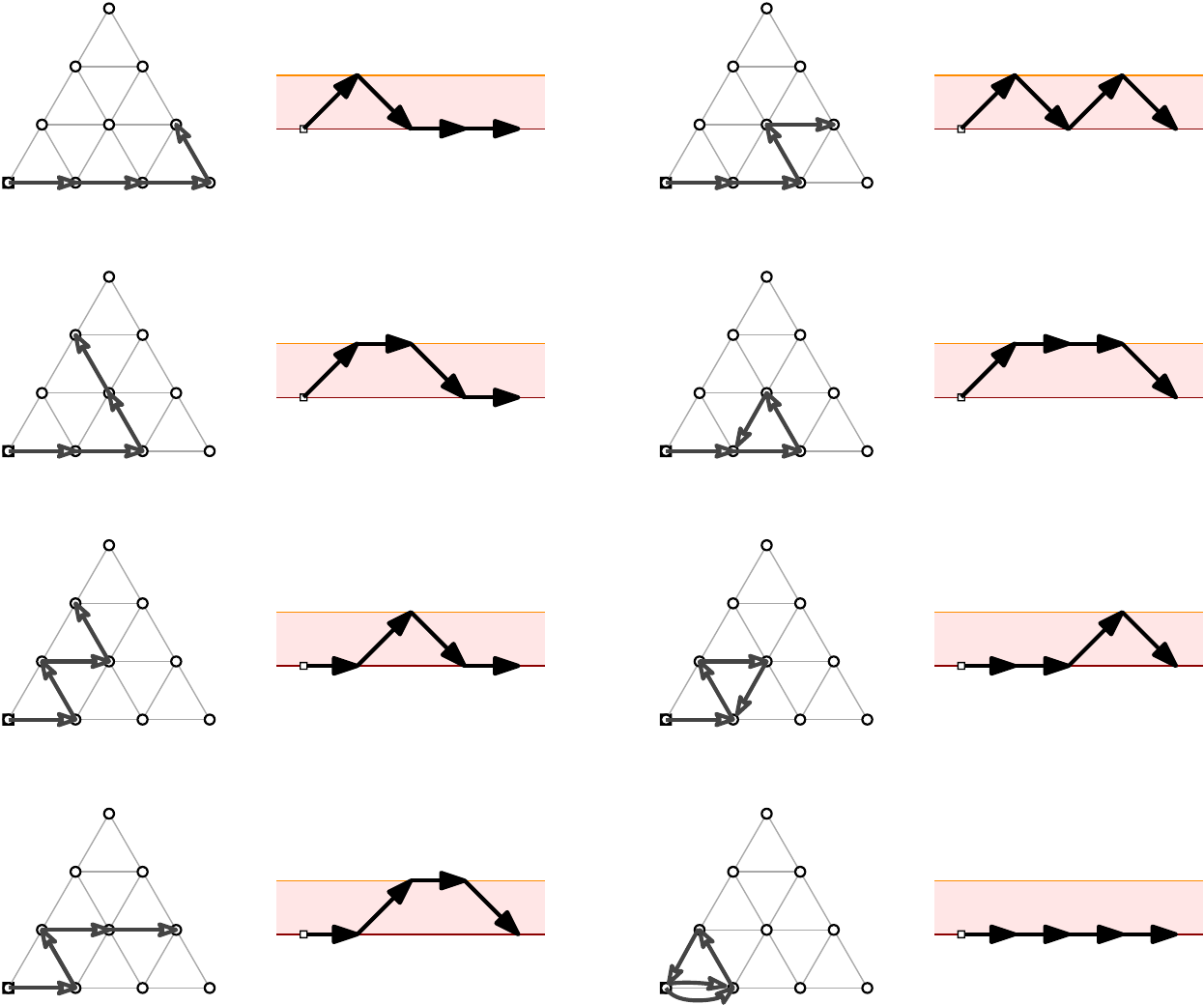}
\end{center}
\caption{Equinumeracy between forward paths of $\T_3$ with length $4$ starting at $\origin$ and Motzkin paths with amplitude bounded by $3$.}
\label{figure:exponential-bijection}
\end{figure}

An illustration of this corollary for $n=4$ is shown by Figure~\ref{figure:exponential-bijection}.

Connections between Motzkin paths and tandem walks (the natural superset of forward paths) are not new. 
Regev~\cite{regev} was the first to notice via an algebraic method that standard Young tableaux with $3$ rows or less and Motzkin paths are counted by the same numbers.
Gouyou-Beauchamps~\cite{gouyouBeauchamps} then found an explanation for this equinumeracy, thanks to the Robinson-Schensted correspondence. Since then, several authors~\cite{eu,eu2,ChyzakYeats,bousquetmelouFusyRaschel} have given new bijections between tandem walks and Motzkin paths, which each have their own ways to be generalized. It should be noted that none of these bijections 
restrict to a bijection between forward paths in $\T_L$ and Motzkin paths with amplitude bounded by $L$.

By comparing Theorem~\ref{theo:mortimerprellberg} and its corollary, one can remark that there is a factor $2^n$ between forward paths in $\T_L$ of length $n$ and generic triangular paths in $\T_L$ of length $n$. This fact was known before Mortimer and Prellberg's article for tandem walks and double-tandem walks (in other words, whenever $L$ is infinite).
Bousquet-Mélou and Mishna~\cite{BoMi10} were the first to notice it and wondered whether there is a combinatorial explanation for this phenomenon. This was solved by Yeats via a convoluted bijection~\cite{yeats2014bijection}. This bijection was subsequently improved by Chyzak and Yeats~\cite{ChyzakYeats} by using the formalism of automata. Again, their bijection does not restrict to the triangular lattice $\T_L$.

\subsection{Outline of the paper}

This paper presents bijections that explain Theorem~\ref{theo:mortimerprellberg}. More precisely, we demonstrate on one hand why the ratio between forward paths and generic paths of length $n$ is $2^n$, and on the other hand, we find several bijections for Corollary~\ref{cor:forward-motzkin}. Combining both results will give different combinatorial 
proofs of Theorem~\ref{theo:mortimerprellberg}.

First, Section~\ref{s:symmetry} concentrates around a symmetry property for the triangular paths: the number of paths starting from a point in $\T_L$ with a fixed sequence of forward and backward steps does not depend on 
the sequence of forward and backward steps. This property, stated by Theorem~\ref{theo:directions}, infers the above-mentioned $1$-to-$2^n$ function between forward paths and triangular paths of length $n$. The proof is based on a convergent rewriting system.

Section~\ref{s:expo} provides a simple inductive proof of the equinumeracy between triangular paths in $\T_L$ and Motzkin paths with amplitude bounded by $L$ (Proposition~\ref{prop:motzkin_inductive}). Furthermore,  we manage to tweak this proof into a bijection which explains Corollary~\ref{cor:forward-motzkin} (see Figure~\ref{figure:Omega}). However, this bijection is highly complex in the sense it is based on an inclusion-exclusion argument and can take an exponential time to be computed.

Almost independently from the previous sections, we describe in Section~\ref{s:scaffolding} a method to build numerous bijections between triangular paths and Motzkin paths of bounded amplitude. To do so, we relate the number of triangular paths starting at any $z \in \T_L$ and the numbers of Motzkin meanders of amplitude bounded by $L$ starting at height $i$ (Theorem~\ref{theo:anywhere}). This proves the existence of an object which we name \textit{scaffolding}, which works 
in much the same way as a finite-state transducer. This enables us to find several parameterized bijections between forward paths and Motzkin paths (Algorithm~\ref{algo:scaffolding1}), which can be extended into bijections between generic triangular paths and bicolored Motzkin paths (Subsection~\ref{ss:bico}). In Subsection~\ref{ss:trapeziums} we give an explicit scaffolding, with simple, albeit numerous transition rules, which has the additional property that it is independent of the size $L$.

Finally, in Section~\ref{s:generalization} we generalize  our results to higher dimensions. The triangular lattice naturally extends to a simplicial lattice, in which the ratio property between forward paths and generic paths (Theorem~\ref{theo:generic_forward}) still holds. More surprisingly, we find a new bijection specifically in dimension $3$. It matches walks using $4$ steps confined within a pyramid with walks using the $4$ cardinal steps returning to the $x$-axis confined in a domain which is the upper half of a square that have been rotated $45^\circ$ (Theorem~\ref{theo:waffle_to_pyramid}). The second family of walks being easier to count than the first one, we find a  formula for the generating function of the pyramidal walks, which was part of an open question from~\cite{MortimerPrellberg}.

The bijections between forward paths and Motzkin paths have been implemented in \texttt{python} and are available at {\url{https://tinyurl.com/yajkqlyv}}.

\section{From forward paths to generic triangular paths}
\label{s:symmetry}

This section describes a one-to-$2^n$ function from the set of forward paths of length $n$ in $\T_L$ to the set of generic paths of length $n$ in $\T_L$.
This is a crucial step in finding a combinatorial proof of Theorem~\ref{theo:mortimerprellberg}.

More precisely, we are going to describe a bijection between {different} sets of paths {where in each set, all paths} have the same sequence of forward and backward steps, which we call {the} \textit{direction vector}.

\begin{Definition}\label{def:direction_vector}
The \emph{direction vector} of a generic path $(\omega_1,\ldots,\omega_n)$ is the finite sequence $(D_1,\dots,D_n)$ where $D_i = F$ if $\omega_i$ is a forward step {and} $D_i = B$ {if $\omega_{i}$} is a backward step.
\end{Definition}

A forward path is then a generic path with direction vector $(F,\ldots,F)$. Many examples of paths along with their direction vectors are shown in Figure~\ref{figure:boolean}.

%
%

\begin{Theorem}\label{theo:directions}
Given $z \in \T_L$ and two sequences $W$ and $W'$ of $\{F,B\}^n$, the set of triangular paths starting from $z$ of direction vector $W$ {is} in bijection with the set of triangular paths starting from $z$ of direction vector $W'$. 
\end{Theorem}

This theorem will be proved in Section~\ref{ss:bij_generic}.

\subsection{Forward and backward paths}

This subsection shows by induction, without a bijection, a particular case of Theorem~\ref{theo:directions} between two direction vectors: $W= (F,\dots,F)$ and $W'=(B,\dots,B)$.
This provides an elementary proof of a weaker result, which enables to understand why the more general theorem works.

\begin{Definition}
A \emph{backward (triangular) path} is a triangular path of direction vector $(B,B,\dots,B)$. In other words, a backward path starting at $z \in \T_L$  is a sequence $(\overline{\sigma_1},\ldots,\overline{\sigma_n})\in \Sb^{n}$ satisfying:
\[\forall k\in\{1,\ldots,n\}, \quad z+\sum_{i=1}^k \overline{\sigma_i} \in \T_L.\]
\end{Definition}


\begin{Theorem} Let $z$ be any point of $\T_L$ and $n \geq 0$. Inside $\T_L$, there are as many \emph{forward} paths of length $n$ starting from $z$ as \emph{backward} paths of length $n$ starting from $z$.
\label{theo:first}
\end{Theorem}

The proof will use the following lemma, which concerns paths with \textit{one} forward step and \textit{one} backward step:

\begin{Lemma} Given a starting point $z$ and an ending point $z'$, there are as many paths of length $2$ from $z$ to $z'$ made of a forward step then a backward step, as paths of length $2$ from $z$ to $z'$ made of a backward step then a forward step.
\label{lemma:length2}
\end{Lemma}

\begin{proof} This lemma is obvious whenever the two steps can be permuted. 

Let us  {first} show that given a forward step $\sigma$ and a backward step~$\overline \tau$ such that $\sigma \neq - \overline \tau$, the path $(\sigma,\overline \tau)$ stays in $\T_L$ from $z$ to $z'$ if and only if the path $(\overline \tau, \sigma)$ stays in $\T_L$ from $z$ to $z'$. For such steps $\sigma$ and $\overline \tau$, there are two possibilities:
\begin{enumerate}
\item \textbf{ $\boldsymbol \sigma$ is a step $\boldsymbol{s_i}$ and $\boldsymbol {\overline \tau}$ is $\boldsymbol{\overline{s_{i+1}}}$.} By cyclic permutation, we can assume that $\sigma = s_1 = e_1 - e_3$ and $\overline \tau = \overline{s_2} = e_1 - e_2$. If $z+\sigma \in \T_L$ and $z + \sigma + \overline \tau \in \T_L$, then $z$ must have a positive $e_2$-coordinate and a positive $e_3$-coordinate. The same property holds if we replace the condition $z+\sigma \in \T_L$ by $z+\overline \tau \in \T_L$. Therefore, we can permute the forward step and the backward step in that case.
\item \textbf{ $\boldsymbol {\overline \tau}$ is a step $\boldsymbol{\overline{s_{i}}}$ and  $\boldsymbol \sigma$ is a step $\boldsymbol{s_{i+1}}$.} Again, we can assume that  $\overline \tau = \overline{s_1} = e_3 - e_1$ and $\sigma = s_2 = e_2-e_1$. Under the assumption that $z + \sigma + \overline \tau \in \T_L$, we need $z$ to have an $e_1$-coordinate at least equal to $2$. In this case, both paths $(\sigma,\overline \tau)$ and $(\overline \tau, \sigma)$ are valid.
\end{enumerate}  

\begin{figure}
\begin{center}
\includegraphics[width = 0.95 \textwidth]{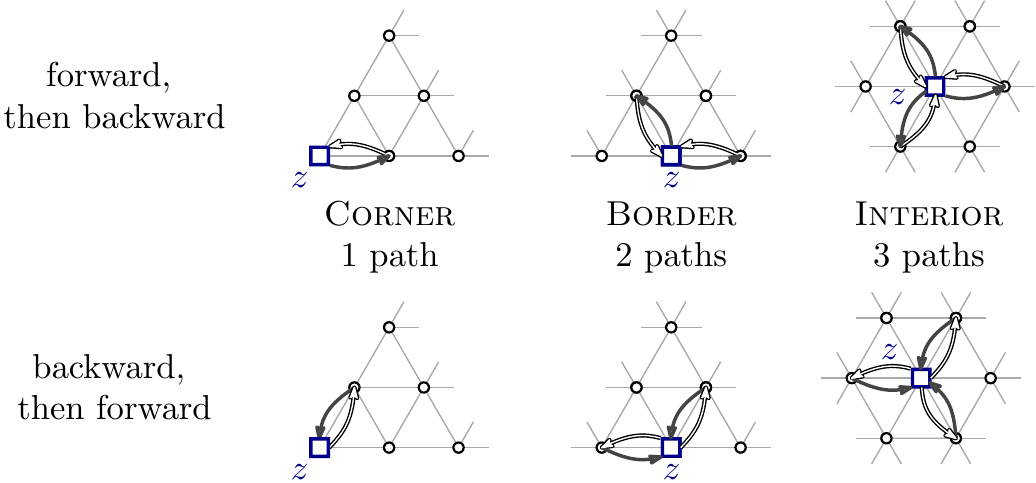}
\end{center}
\caption{All paths of length $2$ returning to their starting point.}
\label{figure:length2}
\end{figure}

It remains to deal with paths satisfying $\sigma = -\overline \tau$. It is equivalent to treat the case $z=z'$.  It is then easy to check that for each possible position of $z$, there are as many paths of length $2$ beginning with a forward step as paths of length $2$ beginning with a backward step,
as summarized by Figure~\ref{figure:length2}.
\end{proof}

\begin{proof}[Proof of Theorem~\ref{theo:first}]
Let $f_n(z)$ be the number of forward paths of length $n$ and starting at $z \in \T_L$, and $b_n(z)$ be the analogue for backward paths.
We wish to prove that $f_n(z) = b_n(z)$ for every $z \in \T_L$ by strong induction on $n\geq 0$. 

For $n=0$ and $n=1$, the property is straightforward. 

Let us assume that the assumption is true for some $n \geq 1$ and $n-1$.
For $z \in \T_L$ we have:
\[
f_{n+1}(z) =  \sum_{ \substack{ \sigma \in \Sf \\ z + \sigma \in \T_L }}  f_n(z + \sigma).
\]
By the induction assumption,
\begin{align*}
f_{n+1}(z) & =  \sum_{ \substack{ \sigma \in \Sf \\ z + \sigma \in \T_L }}  b_n(z + \sigma)  
\\ &= \sum_{ \substack{ \textrm{path of length } 2\\\textrm{from }z\textrm{ to }z' \\ \textrm{made of a forward step}\\ \textrm{then a backward step}}}  b_{n-1}(z').
\end{align*}
We use the induction assumption now for $n-1$, and Lemma~\ref{lemma:length2}:
\begin{align*}
f_{n+1}(z) &= \sum_{ \substack{ \textrm{path of length } 2\\\textrm{from }z\textrm{ to }z' \\ \textrm{made of a backward step}\\ \textrm{then a forward step}}}  f_{n-1}(z')\\
 &= \sum_{ \substack{ \overline \tau \in \Sb \\ z + \overline \tau \in \T_L }}  f_n(z + \overline \tau). \\
  &= \sum_{ \substack{ \overline \tau \in \Sb \\ z + \overline \tau \in \T_L }}  b_n(z + \overline \tau). &\textrm{(by induction)} \\ & = b_{n+1}(z),
\end{align*}
which concludes the induction, and hence the proof.
\end{proof}

\subsection{Bijection between sets of different direction vectors}
\label{ss:bij_generic}


In this subsection, we describe a bijection that proves Theorem~\ref{theo:directions}.

\begin{figure}
\begin{center}
\includegraphics[width = 0.95 \textwidth]{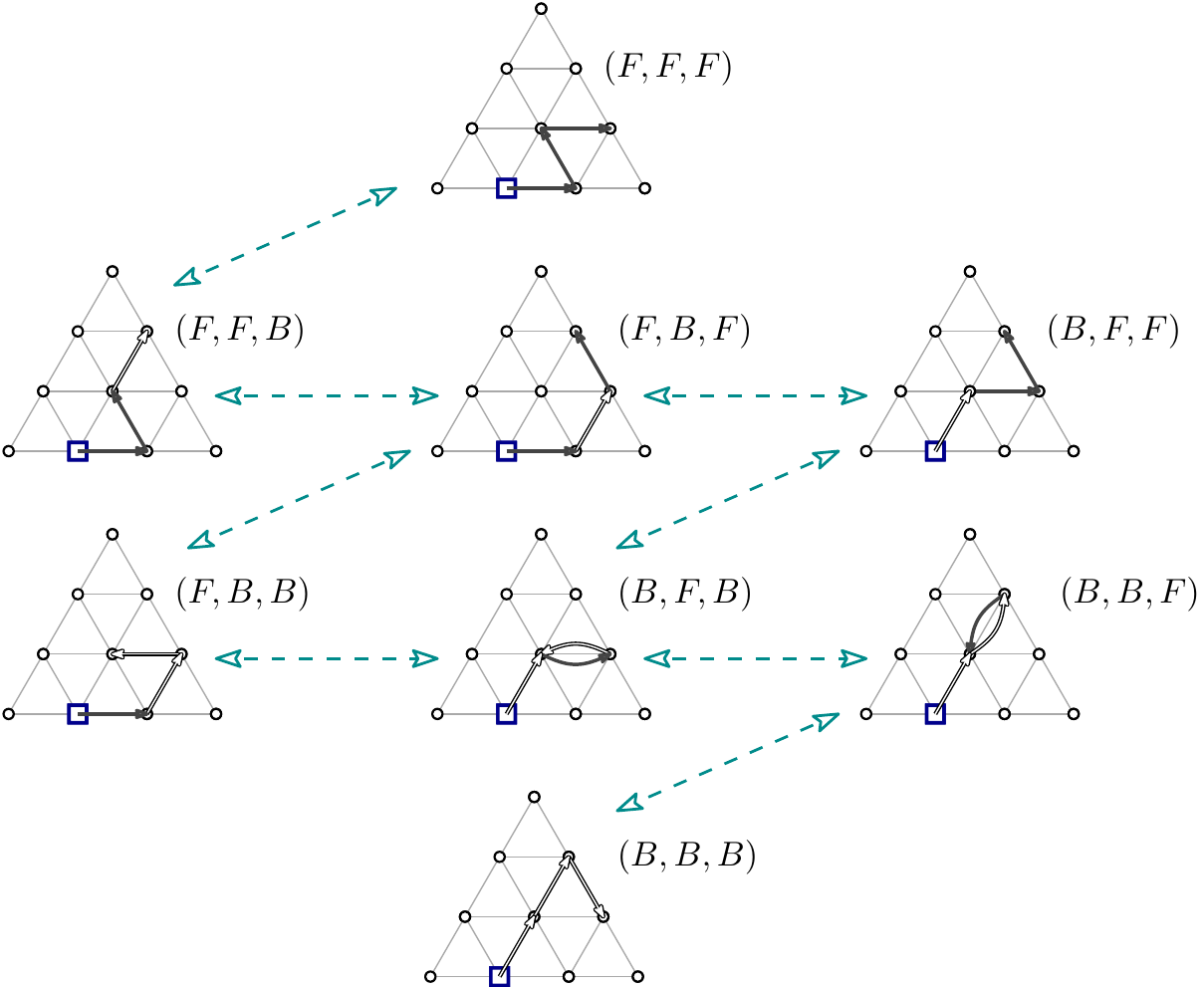}
\end{center}
\caption{The bijections between all direction vectors (arranged as a Boolean lattice) applied to the forward path $(s_1,s_2,s_1)$.}
\label{figure:boolean}
\end{figure}

This bijection consists in combining the elementary operations below, in any possible order, until reaching a path with the desired direction sequences.

\begin{Definition}[Flips] We define here elementary reversible operations on a generic path $(\omega_1,\dots,\omega_n)$. \\
{A \textbf{swap flip}} changes two consecutive steps $\omega_i$ and $\omega_{i+1}$ with respect to the rules:
\begin{align*}
(s_j,\overline{s_k}) \lra (\overline{s_k}, s_j)& \quad \mbox{if} \; j\not=k, \\ (s_k,\overline{s_k}) \lra (\overline{s_{k-1}},s_{k-1}) & \quad  \mbox{otherwise.}
\end{align*}
(Recall that by convention, $s_0 = s_3$.)
This has the effect of doing a flip $(F,B) \lra (B,F)$ in the direction vector.

{A} \textbf{last-step flip} changes the direction of the last step $\omega_n$ thanks to the rule:
\[s_i \lra \overline{s_{i-1}}\]

\label{def:flips}
\end{Definition}

For example, if we wish to bijectively transform the path $(\bsc,\bsc,\bsb)$ into a path of direction vector $(F,B,B)$, we use the following flips (cf Figure~\ref{figure:boolean}):
\begin{align*}
(\bsc,\bsc,\bsb) \quad & \underset{  \bsb \rightarrow s_3  }\longleftrightarrow \quad (\bsc,\bsc,s_3) \quad  \underset{ (\bsc,s_3) \rightarrow (s_1,\bsa)   }\longleftrightarrow \quad  (\bsc,s_1,\bsa) \\
 & \underset{  \bsa \rightarrow s_2  }\longleftrightarrow \quad  (\bsc,s_1,s_2) \quad   \underset{ (\bsc,s_1) \rightarrow (s_1,\bsc)   }\longleftrightarrow \quad  (s_1,\bsc,\bsa).
\end{align*}

Note that swap flips give a constructive proof to Lemma~\ref{lemma:length2}.

\begin{proof}[Proof of Theorem~\ref{theo:directions}]

We want to prove that successive flips induce a well-defined bijection between sets of triangular paths with different direction vectors. To do so, we have to establish the following points.

\begin{enumerate}
\item \textbf{The flips are well defined.} 

In other words, we want to show that a flip does not make a path of $\T_L$ go outside $\T_L$.

For  flips swapping steps $s_i$ and $\overline{s_j}$ such that $s_i \neq - \overline{s_j}$, we showed in the proof of Lemma~\ref{lemma:length2} that a forward step and a backward step can commute under the condition that the two steps are not opposite. 

The swap flip $(s_1,\bsa) \lra (\bsc,s_3)$ is also well-defined because $s_1$ and $\bsc$ have both a negative $e_3$-coordinate. Therefore, the position of the point just before the flip must have a positive $e_3$-coordinate. One can safely  apply $s_1$ or $\bsc$. 

Similar arguments hold for the other swap flips, and for last-step flips.

\item \textbf{Each flip is bijective.}

This is clear from the definition of the flips.

\item \textbf{Given two sequences $W$ and $W'$ of $\{F,B\}^n$, one can transform any path with direction vector $W$ into a path of direction vector $W'$ by successive flips.}

If $W$ and $W'$ have the same number of $B$'s, then we can use swap flips to transform a walk of direction vector $W$ into one of direction vector $W'$.

Otherwise, we can increment (resp. decrement) the number of $B$'s of the direction vector by putting a forward step (resp. a backward step) at the end of the walk using successive swap flips, then changing the direction of this last step using a last-step flip. We rinse and repeat until obtaining the desired number of $B$'s, then use swap flips as above.

\item \textbf{If two different sequences of flips lead to triangular paths $p$ and $p'$ that share a same direction vector, then $p = p'$.
}\label{item:unique}
\end{enumerate}
The proof of the last point is postponed {until} the next subsection (Proposition~\ref{prop:tiling}).
\end{proof}

In particular, Theorem~\ref{theo:directions} gives a bijective proof of Theorem~\ref{theo:first}. If we wish to make it explicit, we can write an algorithm that chooses a specific sequence of flips that transforms an $(F,\dots,F)$ direction vector into a $(B,\dots,B)$ vector.

\begin{Corollary} Given $z \in \T_L$ and an integer $n$, Algorithm~\ref{algo1} forms a bijection between forward paths of length $n$ starting at $z$ and backward paths of length $n$ starting at $z$. This bijection depends neither on the length $L$ of the triangular lattice, nor on the position of the starting point $z$.
\end{Corollary}

\begin{algorithm}[caption={Bijection between forward paths and backward paths (for \textit{flips}, see Definition~\ref{def:flips}).}, label={algo1}]
input: a forward path p
output: a backward path p
n $\gets$ length of p;
for i from 1 to n 
do make a last-step flip on p[i];
   for j decreasing from n-1 to i 
   do	make a swap flip between p[j] and p[j+1];  
\end{algorithm}

\begin{Remark} Algorithm~\ref{algo1} also transforms (in a bijective manner) a backward path into a forward path. Thus, if we apply twice Algorithm~\ref{algo1} to a forward path, we also obtain at the end a forward path. Therefore, assuming that the uniqueness claimed in Item~\eqref{item:unique} in the proof of Theorem~\ref{theo:first} holds (and it does), the two forward paths must be the same: Algorithm~\ref{algo1} is in fact an involution. 
\end{Remark}

\subsection{Description of the bijection in terms of folded paths}

This section presents the bijection of Theorem~\ref{theo:directions}  in a more symmetric fashion. The last-step flip, which we defined in Definition~\ref{def:flips}, can be actually seen as a disguised swap flip, under the condition that the path is extended to what we call a \textit{folded path}.

\begin{Definition}[Folded paths]
Given a generic path $\omega = (\omega_1,\dots,\omega_n) \in (\Sf\cup \Sb)^n$, we define the \emph{folding} of $\omega$ as the path 
$$\dba{\omega}=(\omega_1,\ldots,\omega_n,-\omega_n,\ldots,-\omega_1).$$
Such paths are said to be \emph{folded}.
\end{Definition}


Let us denote by $\mathcal S_n$ the tilted square lattice
\[ \mathcal S_n = \{ (i,j) \, \in \, \N \times \N \quad : \quad |i| + |j| \leq n \}. \]
We will geometrically represent folded paths of length $2n$ as labeled walks on $\mathcal S_n$ starting at $(-n,0)$. To construct the walk on $\mathcal S_n$, we replace every forward step  by a North-East step $(+1,+1)$, and every backward step by a South-East step $(+1,-1)$. Moreover, these North-East and South-East steps will carry labels, which are the steps of $\Sf \cup \Sb$ from which they originate. For example, the folding of the path $(s_1,\bsc,\bsa)$ is represented on the left of Figure~\ref{figure:diamond}.

\begin{figure}
\begin{center}
\includegraphics[width = 0.95 \textwidth]{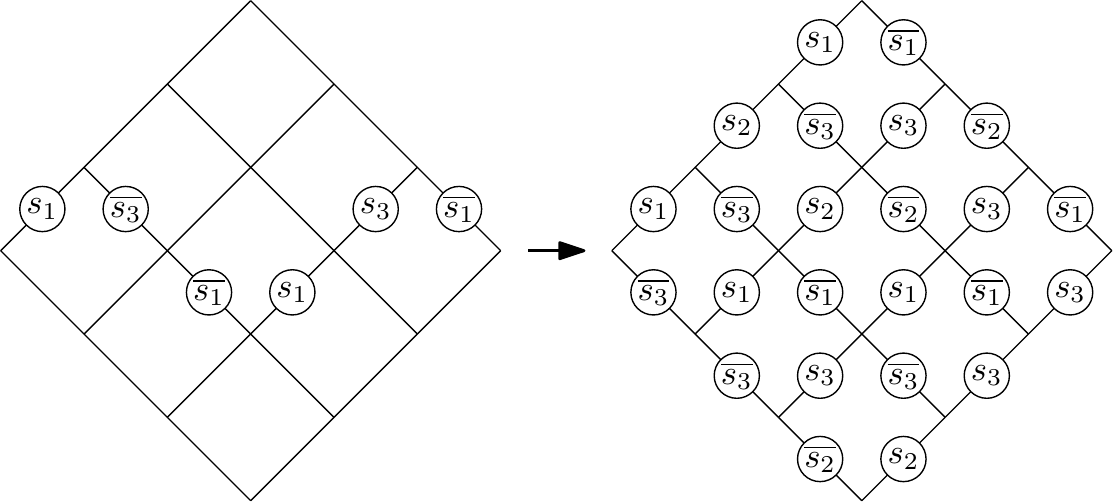}
\end{center}
\caption{The geometric representation of the bijection}
\label{figure:diamond}
\end{figure}

Now, we are going to emulate the effect of swap flips (see Definition~\ref{def:flips}) on these walks. More precisely, we view $\mathcal S_n$ as a square of size $n \times n$ which can be filled out with $1 \times 1$ square tiles of $9$ types (see Figure~\ref{figure:rules}). The four sides of the $9$ allowed tiles are labeled with elements of $\Sf \cup \Sb$ such that the pairs formed by the two top labels and the two bottom labels correspond to a commutation rule described in Definition~\ref{def:flips}. 

The tiling of $S_n$ proceeds as follows. We begin with the labels given by a folded path. 
Then, we place copies of the tiles of Figure~\ref{figure:rules} in such a way that the two top labels or the two bottom labels match (like a domino) with labels which were already in $S_n$.
Eventually, we obtain an alternative description of the bijection of Theorem~\ref{theo:first}, and thus the required uniqueness:

\begin{figure}
\begin{center}
\includegraphics[scale=1.1]{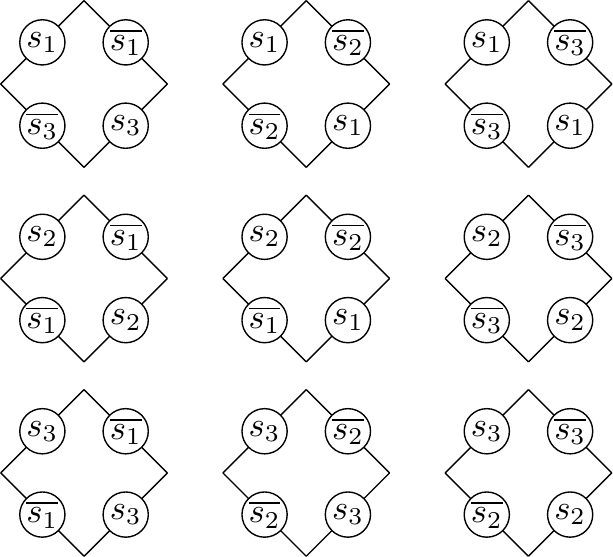}
\end{center}
\caption{The $9$ possible tiles}
\label{figure:rules}
\end{figure}

\begin{Proposition} Let $\dba \omega$ be the folding of a triangular path $\omega$ of length $n$, which we  embed in the tilted square lattice $\mathcal S_n$ as described above.

There is a unique way to tile $\mathcal S_n$ with the $9$ tiles of Figure~\ref{figure:rules} while preserving the labels of $\dba \omega$.

Furthermore, let us fix a sequence $W = (W_1,\dots,W_n)$ of $\{F,B\}^n$. The path of direction vector $W$ which corresponds to $\omega$ under the bijection of Theorem~\ref{theo:directions} is defined by the sequence of labels obtained by following the walk in $S_n$ whose $k$-th step is North-East if $W_k = F$ or South-East if $W _k = B$.

\label{prop:tiling}
\end{Proposition}

\begin{Example}
Let us consider the path $(s_1,\bsc,\bsa)$, represented in Figure~\ref{figure:rules} (left). The unique corresponding tiling is displayed on the right of the figure.

If we want the path of direction vector $(B,F,F)$ corresponding to $(s_1,\bsc,\bsa)$, then we have to read labels from the walk going SE, NE, NE (in this order). We find $(\bsc,s_1,s_2)$.
\end{Example}

\begin{proof}[Proof of Proposition~\ref{prop:tiling}]
The existence and the uniqueness of the tiling are proved by induction. We just have to notice that every pair $(\sigma,\overline \tau)$ with $\sigma \in \Sf$ and $\overline \tau \in \Sb$ appears once among the top labels of the $9$ tiles, and every pair $(\overline \tau, \sigma)$ appears also once among the bottom labels. We have no choice {in} how to place new tiles: the tiling is automatic and unambiguous.

To connect the tiling with the bijection of Theorem~\ref{theo:first}, note that:
\begin{itemize}
\item A swap flip at positions $k$ and $k+1$ can be emulated by positioning a tile along the $k$-th and the $(k+1)$-th step  \textit{and} by symmetrically placing a second tile along the $(2n-k+1)$-th and the $(2n-k)$-th step.
\item A last-step flip can be emulated by positioning a tile on the vertical axis of $S_n$. 
\end{itemize}
One thus recovers what we described in previous subsection.
\end{proof}

As a consequence, in view of the vertical symmetry of the tiling, one can describe the bijection of Theorem~\ref{theo:first} uniquely in terms of swap flips -- as claimed at the beginning of this subsection.

\begin{Corollary}
The folded paths of direction vector $(F,\dots,F,B,\dots,B)$ are in bijection with the folded paths of direction vector $(B,\dots,B,F,\dots,F)$ via successive uses of swap flips. 
\end{Corollary}

\section{A first bijection between forward paths and Motzkin meanders}
\label{s:expo}

In this section, we provide two proofs of Corollary~\ref{cor:forward-motzkin}: the first one uses an induction and is elementary, the second one is based on a recursive bijection which is derived from the first proof.

\subsection{Recursive proof of the equinumeracy}
\label{ss:equinumerosity}

The following proposition links Motzkin meanders and forward paths starting from the border of $\T_L$.


%



\begin{Proposition}\label{prop:motzkin_inductive} For any $n \geq 0$ and $L > 0$, let $f_n(z)$ be the number of forward paths in $\T_L$ of length $n$ starting at $z$, and $\Mo{L}{n}{\ell}$ the number of Motzkin meanders of length $n$ starting at height $\ell$ and with an amplitude bounded by $L$ (see  Subsection~\ref{ss:motzkin} for the definitions).

Then, we have the formula
\[f_n(\origin+\ell s_1)=\sum_{i=0}^{\ell} \Mo{H}{n}{i},\]
for $\ell \in \{0,\ldots,\lfloor L/2\rfloor\}$.
\end{Proposition}

As a particular case $\ell=0$ of the result above, 
we recover the statement of Corollary~\ref{cor:forward-motzkin}.


\begin{Example} Figure~\ref{figure:forward-motzkin} corroborates Proposition~\ref{prop:motzkin_inductive} with $n=3$, $L = 3$,  and $\ell = 1$: numbers agree ($8$ on each side). Remark that if $L$ is larger ($L \geq 4$), the forward path $s_1 s_1 s_1$ will be added on the left, and the Motzkin meander $\nearrow, \searrow, \searrow$ on the right.
\end{Example}

\begin{figure}
\begin{center}
\includegraphics[width = .9 \textwidth]{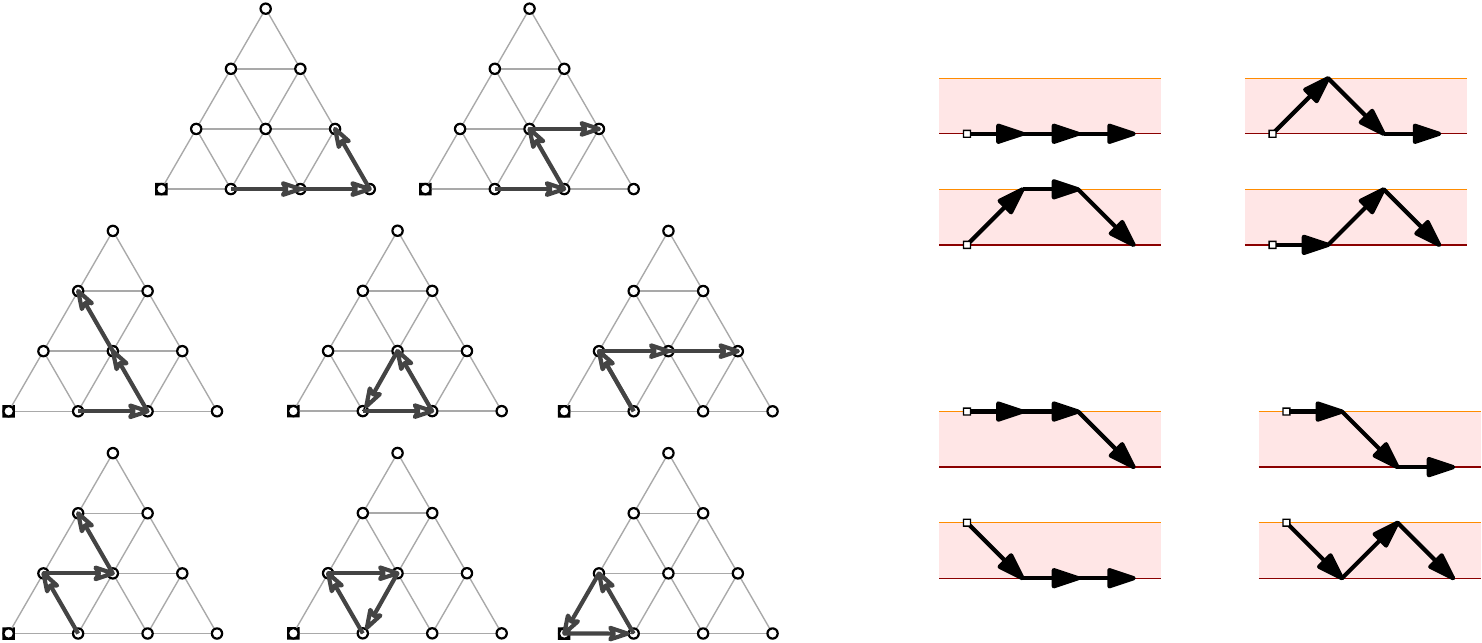}
\end{center}
\caption{\textit{Left.} $8$ forward paths of length $3$ starting from $\origin + s_1$ in $\T_3$. \textit{Right.} $8$ Motzkin meander of length $3$ and amplitude bounded by $L = 3$: four of them begin at height $0$, the remaining four begin at height $1$.}
\label{figure:forward-motzkin}
\end{figure}

\begin{proof}[Proof of Proposition~\ref{prop:motzkin_inductive}.] Let us introduce the notation $\f_n( \ell)=f_n(\origin + \ell s_1)$, with the convention that $\f_n(\ell)=0$ for $\ell<0$. Let us also write $\D \f_n(\ell)=\f_n(\ell)-\f_n(\ell-1)$, and $H = \lfloor L/2\rfloor$.

Note that the numbers of Motzkin meanders  $\Mo{H}{n}{\ell}$ satisfies the obvious recurrences 
\begin{align*} \Mo{H}{n}{\ell} & = \Mo{H}{n-1}{\ell-1} + \Mo{H}{n-1}{\ell} + \Mo{H}{n-1}{\ell+1} & \textrm{for }\ell \in \{1,\dots,H-1\}, \\
\Mo{H}{n}{0} & = \Mo{H}{n-1}{0} + \Mo{H}{n-1}{1}, \\
\Mo{H}{n-1}{H} & =  \left\{ \begin{array}{ll}\Mo{H}{n-1}{H-1} + \Mo{H}{n}{H} & \textrm{ if }L\textrm{ is odd}\\ \Mo{H}{n-1}{H-1} & \textrm{ if }L\textrm{ is even}\\ \end{array}\right. ,
\end{align*}
for $n \geq 1$. The proof is completed whenever we find the same recurrences for $\Delta g_n(i)$. The reader can refer to Figure~\ref{figure:recursion_explanation} as a visual support for what follows.


\begin{figure}
\begin{center}
\includegraphics[width = .9 \textwidth]{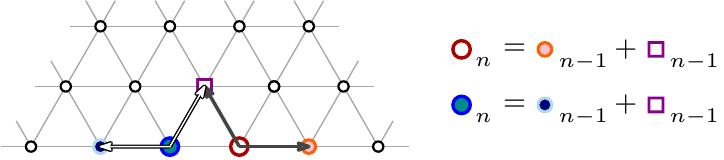}
\end{center}
\caption{Explanation of Equations \eqref{eq:firstpiece} and \eqref{eq:secondpiece} in generic case. A dot with a subscript~$n$ represents the number of forward paths of length $n$ starting from this point (which is, by Theorem 2, also the number of backward paths). }
\label{figure:recursion_explanation}
\end{figure}

For any $\ell\in\{1,\ldots,H-1\}$, starting from $\origin +\ell s_2$, the only possible forward steps are $s_1$ and $s_2$, so that
\begin{align}
\f_n(\ell)&=f_{n-1}(\origin+ \ell s_1 + s_2) + f_{n-1} (\origin + \ell s_1 +s_2) \nonumber \\
& = \f_{n-1}(\ell+1)+f_{n-1}(\origin +\ell s_1+s_2). \label{eq:firstpiece}
\end{align}
We now count backward paths starting from $\origin+ (\ell-1) s_1$. By Theorem~\ref{theo:first}, if $b_n(z)$ is the number of backward paths of length $n$ starting at $z$, we have $f_n(z)=b_n(z)$ for every $z \in \T_L$.
In particular, $\f_n(\ell - 1) = b_n(\origin + (\ell - 1) s_1)$. Since only possible backward steps from $\origin + (\ell - 1) s_1$ are $\bsa$ and $\bsc$, we have for any $\ell\in\{1,\ldots,H-1\}$,
\begin{align}
\f_n(\ell-1)&=
b_{n-1}(\origin +(\ell-1) s_1 + \bsa )+b_{n-1}(\origin + (\ell-1) s_1 + \bsc) \nonumber \\
&= f_{n-1}(\origin +(\ell-1) s_1 + \bsa )+f_{n-1}(\origin + (\ell-1) s_1 + \bsc) \nonumber \\
& = \f_{n-1}(\ell-2) + f_{n-1}(\origin+ \ell s_1 + (\bsc - s_1) ) \nonumber \\
&=\f_{n-1}(\ell-2) + f_{n-1}(\origin +\ell s_1+s_2) \label{eq:secondpiece}.
\end{align}
(Note that the case $\ell = 1$ is correctly handled since by convention, $\f_{n-1}(-1)=0$.) 
Combining \eqref{eq:firstpiece} and \eqref{eq:secondpiece}, we deduce that for $\ell\in\{1,\ldots,H-1\}$,
\[\f_n(\ell)-\f_n(\ell-1)=\f_{n-1}(\ell+1)-\f_{n-1}(\ell-2),\]
and hence
\[\D \f_n(\ell)=\D \f_{n-1}(\ell-1)+\D \f_{n-1}(\ell)+\D \f_{n-1}(\ell+1).\]


As for $\ell=0$, we straightforwardly have
\begin{align*}
\D \f_{n}(0)&=\f_n(0)=\f_{n-1}(1)\\
&=\D \f_{n-1}(0)+\D \f_{n-1}(1).
\end{align*}

\textbf{(i) Let us first assume that $L=2H+1$ is odd.} 
Then,
using a symmetry through the plan of equation $x_1 = x_3$ ($x_1$ being the coordinate in $e_1$ and $x_3$ the one in $e_3$), we have $f_{n-1}( \origin + H s_1   )=b_{n-1}(\origin + (H+1) s_1 )$. 
By Theorem~\ref{theo:first}, it translates $\f_{n-1}(H) = \f_{n-1}(H+1)$.   Thus, $\D \f_{n-1}(H+1)=0$, and
$$\D \f_{n}(H)=\D \f_{n-1}(H-1)+\D \f_{n-1}(H).$$
It follows that $(\D \f_{n}(\ell))_{0\leq \ell \leq H}$ satisfies the following recursion
$$
\begin{pmatrix}
\D \f_n(0)\\
\D \f_n(1)\\
\vdots\\
\vdots \\
\vdots\\
\D \f_n(H)\\
\end{pmatrix}
=
\begin{pmatrix}
1&1&0&\cdots&\cdots&0\\
1&1&1&\ddots&&\vdots\\
0&1&1&1&\ddots&\vdots\\
\vdots&\ddots&\ddots&\ddots&\ddots&0\\
\vdots&&\ddots&1&1&1\\
0&\cdots&\cdots&0&1&1\\
\end{pmatrix}
\begin{pmatrix}
\D \f_{n-1}(0)\\
\D \f_{n-1}(1)\\
\vdots\\
\vdots \\
\vdots\\
\D \f_{n-1}(H)\\
\end{pmatrix}
$$
which is the same recursion that we saw for $(\Mo{H}{n}{\ell})_{0\leq \ell \leq H}$. Since the base cases agree ($\D \f_0(\ell) = \Mo{H}{0}{\ell} = 0$ for $\ell > 1$, and $\D \f_0(0) = \Mo{H}{0}{0} = 1$),
 we have the equality $\Mo{H}{n}{\ell}=\D\f_n(\ell)$, and the result directly follows.

 \textbf{(ii) Let us now assume that $L=2H$ is even.} Always thanks to the symmetry with respect the plane $x_1 = x_3$, we have $\f_{n-1}(H-1)=\f_{n-1}(H+1)$, so that $\D \f_{n-1}(H+1)+\D \f_{n-1}(H)=0$, and
\[\D \f_n(H)=\D \f_{n-1}(H-1).\]
It follows that $(\D \f_n(\ell))_{0\leq \ell \leq H}$ satisfies the following recursion
$$
\begin{pmatrix}
\D \f_n(0)\\
\D \f_n(1)\\
\vdots\\
\vdots \\
\vdots\\
\D \f_n(H)\\
\end{pmatrix}
=
\begin{pmatrix}
1&1&0&\cdots&\cdots&0\\
1&1&1&\ddots&&\vdots\\
0&1&1&1&\ddots&\vdots\\
\vdots&\ddots&\ddots&\ddots&\ddots&0\\
\vdots&&\ddots&1&1&1\\
0&\cdots&\cdots&0&1&0\\
\end{pmatrix}
\begin{pmatrix}
\D \f_{n-1}(0)\\
\D \f_{n-1}(1)\\
\vdots\\
\vdots \\
\vdots\\
\D \f_{n-1}(H)\\
\end{pmatrix}.
$$
We thus recover the recursion of $(\Mop{H}{n}{\ell})_{0\leq \ell \leq H}$, and we conclude like above.
%
\end{proof}

\subsection{Exponential bijection}

We now convert the argument of Subsection~\ref{ss:equinumerosity} to a bijection, albeit one which is defined recursively and takes non-linear time to apply. 


We fix in this section the length $L$ of the triangular lattice $\T_L$, and $H$ the semi-length: $H = \lfloor L/2\rfloor$.

Let $G_{n}(k)$ be the set of forward paths of length $n$ starting at $\origin + k s_{1}$ and let $M_{n}(k)$ be the set of Motzkin meanders of length $n$ starting at height $k$ and having amplitude bounded by $L$.
  
It follows from Proposition \ref{prop:motzkin_inductive} that $|M_{n}(k)|=|G_{n}(k)|-|G_{n}(k-1)|$. 

To show this bijectively, we will recursively define a sequence of bijective functions $\Omega_{n,k}:G_{n}(k)\to M_{n}(k)\cup G_{n}(k-1)$ for $n\in\mathbb{N}$ and $k\in[0,H]$. This will use the bijection of Theorem \ref{theo:directions} between triangular paths with different direction vectors. In particular, we will use this in the special cases sending paths with some direction vector $W$ of length $n$ to paths with direction vector $(F,\dots,F)$. We denote this function by $W_{n}$ -- this forms a bijection when the domain is restricted to those paths with some explicit direction vector.

\begin{figure}
\begin{center}
\includegraphics[width = \textwidth]{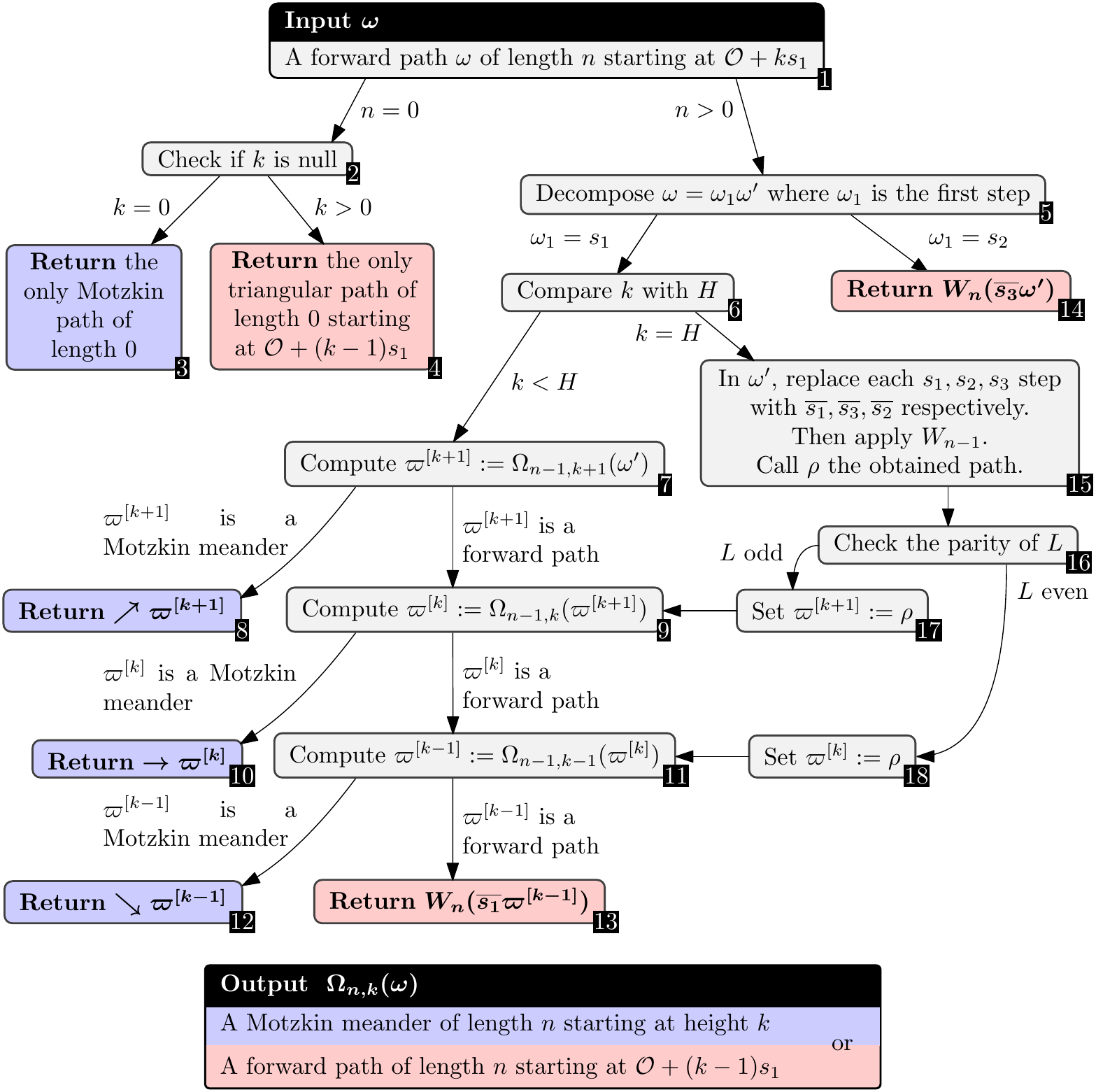}
\end{center}
\caption{Algorithm computing  $\Omega_{n,k}(\omega)$ where $\omega$ is a path of length $n$ starting at $\origin + k s_1$}
\label{figure:Omega}
\end{figure}

\begin{Theorem}
Let $k$ and $n$ be two integers with $k \leq H$. The function
$\Omega_{n,k}$, defined by Figure~\ref{figure:Omega},
is a bijection  from $G_{n}(k)$ to $M_{n}(k)\cup G_{n}(k-1)$, where $G_{n}(k)$ is the set of forward paths of length $n$ starting at $\origin +k s_{1}$, and $M_{n}(k)$ is the set of Motzkin meanders of length $n$ starting at height $k$ and having amplitude bounded by $L$.
\label{theo:Omega}
\end{Theorem}

\begin{proof}\textbf{
1. Let us show that the map is well-defined, i.e. its image is included in $M_{n}(k)\cup G_{n}(k-1)$.
}

It is quite straighforward, except maybe two points:
\begin{itemize}
\item
Why is path $\rho$ from block $15$ of Figure~\ref{figure:Omega} a valid triangular path of $\T_L$? By replacing $s_{1},s_{2},s_{3}$ steps with $\bsa, \bsc, \bsb$, path $\omega'$ undergoes a vertical reflection about the vertical midline of $\T_L$. Thus, $\omega'$ is transformed into a backward path starting at $\origin + H s_1$ (if $L$ is odd) or at $\origin + (H-1) s_1$ (if $L$ is even). Applying $W_{n-1}$ makes it a forward path, which is $\rho$, that belongs to $G_{n-1}(H)$ (when $L$ is odd) or $G_{n-1}(H-1)$ (when $L$ is even). 

\item When $k = H$, it is impossible to output a Motzkin meander starting at height $H$ and beginning by a $\nearrow$ step. So the amplitude of every meander in the image is bounded by $2H+1$. Moreover, when $L$ is even and $k = H$, the returned meanders cannot begin by a horizontal step, which explains why they have amplitude bounded by $L = 2H$.
\end{itemize}

\noindent\textbf{ 2. Let us show by induction on $n$ that $\Omega_{n,k}$ is a bijection for every $k \geq 0$. }
 
The case $n=0$ is clear.

Let $n$ be a positive integer. If the image is a Motzkin meander beginning by $\nearrow$ (resp. $\rightarrow$, $\searrow$), then the algorithm must end at block $8$ (resp. $10$, resp. $12$). This covers all Motzkin paths of $M_n(k)$ (or $M'_n(k)$). Then we can bijectively recover the original path $\omega$ by following the arrows backwards up to block $1$. 
In fact, all the arrows are reversible, notably because of the induction hypothesis. 
There is no ambiguity from blocks $9$ and $12$ (where there are \textit{a priori} two possible ingoing arrows) because one can only go to block $7$ and $9$  if $k < H$. In the contrary case where $k = H$, one have to go to the right side of the diagram (blocks 17 and 18).  

If the image is in $G_{n}(k-1)$, then the algorithms ends either to block 13 or to block 14. Since $W_{n}$ is a bijection from paths with direction vector $(B,F,F,\dots,F)$ to forward paths in $G_{n}(k-1)$, we can recover the preimage under $W_n$. If this preimage begins by $\bsa$, then the algorithm actually ended at block $13$; if it begins by $\bsc$, the algorithm ended at block $14$. At this point, we can use the above reasoning to go backwards to the root of the decision tree and find $\omega$. Thus, we prove that $\Omega_{n,k}$ is a bijection.
\end{proof}

When $k=0$, Theorem~\ref{theo:Omega} provides a bijection between forward paths and Motzkin paths of bounded amplitude.  
Go back to Figure~\ref{figure:exponential-bijection} for examples:  each forward path is put aside its image under $\Omega_{3,0}$.

Thus, at this point, we have answered Mortimer and Prellberg's open question (Theorem~\ref{theo:mortimerprellberg}). Indeed, starting from a bicolored Motzkin path $m$ (let us say in black and white) of length $n$ and of amplitude bounded by $L$, we can construct a direction vector $W$ from it: write $F$ for each black step; $B$ for each white step. Then, we compute $\Omega_{n,0}(m)$, which is a forward path. Finally, we use the bijection from Theorem~\ref{theo:directions} to transform the forward path into a triangular path of direction vector $W$.

Finally, let us  discuss about the complexity of the algorithm. If $c(n,k)$ denotes the worst-case complexity of $\Omega_{n,k}$, then we can derive from Figure~\ref{figure:Omega} the (rough) upper bound
\[c(n,k) \leq c(n-1,k+1) + c(n-1,k) + c(n-1,k-1) + n^2.\]
(The $n^2$ term reflects the complexity of the function $W_{n-1}$ appearing in block $15$.) Then, by a simple induction, one can see that $c(n,k) \leq \Mo H n k + O(n^3)$ where $\Mo H n k$ is the number of Motzkin meanders of length $n$ starting at height $k$ and having amplitude bounded by $L$. Since $\Mo H n 0$ is $O(3^n)$, we deduce that the complexity of $\Omega_{n,0}$ is bounded by an exponential in $n$. However, we do not know if this bound is tight. Experimentally, we have observed that the complexity of the algorithm has a large standard deviation when the input is randomly chosen: in most cases, the complexity is linear in $n$ (in terms of running time and the number of recursive calls) but sometimes the complexity seems to be quadratic in $n$.

\section{Many other bijections}
\label{s:scaffolding}

In the previous section, we described a bijection between forward paths and Motzkin paths of bounded amplitude. However, the definition being recursive, the computation of an image takes \textit{a priori} a long time, and its description lacks some clarity. 

This section proposes a new way to define bijections between forward paths and Motzkin paths.
Such bijections will have a double advantage. First, they only require linear time to compute. Second, these bijections are parameterized: each one of them comes with a specific metadata (which we name \textit{scaffolding}), making them all different. 

\subsection{Profile}

We start to define a integer vector for each point of $\T_L$:

\begin{Definition}[Profile]
Let $z = i e_1 + j e_2 + k e_3$ be any point of $\T_L$. The \emph{profile} of $z$ is the vector $(p_0(z),\dots,p_H(z))$ where $H = \left \lfloor { \frac L 2 } \right\rfloor$ and $p_0(z),\dots,p_H(z)$ is the first half of the coefficients of the polynomial
\[ \frac{(1 - x^{i+1})(1 - x^{j+1})(1- x^{k+1})}{(1-x)^2} =  p_0(z) + p_1(z) x + \dots + p_H(z) x^H + \dots + p_{L+1}(z) x^{L+1}.\]
\label{def:profile}
\end{Definition}

\begin{figure}
\begin{center}
\includegraphics[scale = 1.5]{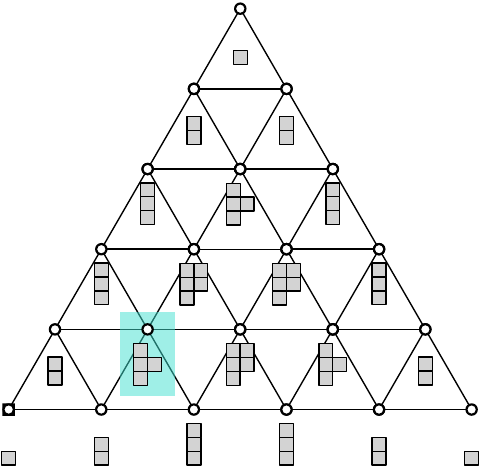}
\end{center}
\caption{A cell representation of $\T_5$. The enlighten zone corres\-ponds to point $e_1 + e_2 + 3 e_3$.}
\label{figure:profile}
\end{figure}

\begin{Example}
Fix $L=5$. The profile of any corner  of $\T_5$ (that is $5e_1$, $5 e_2$ or $5 e_3$) is $(1,0,0)$ since the corresponding polynomial is $(1-x^6)$ (regardless of the corner). The profile of the point $ e_1 + e_2 + 3 e_3$ is $(1,2,1)$, which can be found by expanding the polynomial $(1-x^2)^2(1-x^4)/(1-x)^2  =  1 + 2 x + x^2 - x^4 - 2 x^5 -x^6.$
\end{Example}

Note that one can also extend the definition of profile for  points $i e_1 + j e_2 + k e_3$ where $i=-1$ or $j=-1$ or $k = -1$. Even if they are not in $\T_L$, we can see that the polynomial $\frac{(1 - x^{i+1})(1 - x^{j+1})(1- x^{k+1})}{(1-x)^2}$ is null for such points, so by convention, we can define the profile as the null vector $(0,\dots,0)$.
It will be useful to deal with border cases.

It is convenient to represent the profiles as sets of square cells. 

\begin{Definition}[Cell representation]
A \emph{cell representation of a point $z$} is a finite subset $\C(z)$ of $\Z^2$ satisfying $\left|\{ \ell \ : \ (f,\ell) \in \C(z)\}\right| = p_f(z)$ for every $f \in \{0,\dots,H\}$. A \emph{cell representation of $\T_L$} is a family $\C=(\C(z))_{z \in \T_L}$ of cell representations of points of $\T_L$. The \emph{height} of a cell $c=(f,\ell)$ is defined as  $h(c)=f$. 
\label{def:cell_representation}
\end{Definition}

The profile of every point $z$ is then illustrated by the cell representation $\C(z)$: for every $(f,\ell) \in \C(z)$, a square is placed at coordinates $(\ell,f)$.\footnote{We swap the two coordinates so that $f$ (which stands for \textit{floor}) corresponds to the height of a cell, consistent with the fact that $f$ represents the height in a Motzkin path.} 
For example, as shown by Figure~\ref{figure:profile}, the cell representation of $e_1 + e_2 + 3 e_3$ in $\T_5$ (whose profile is $(1,2,1)$, as mentioned above) can be represented as three rows of squares: the first (bottom) and the third (top) rows have $1$ square each while the central row has $2$ squares.

It is not obvious from Definition~\ref{def:profile} that we always have $p_f(z) \geq 0$, and hence that a cell representation of $\T_L$ exists for every $L \in \N$. However a cell representation of $\T_L$ will be explicitly given by Proposition~\ref{prop:cell_rep}, proving the non-negativity of the components of a profile.


The next lemma establishes some identities about the profile.

\begin{Lemma}Let $z$ be in $\T_L$.
Then for $i \in \{1,\dots,H-1\},$ the identities
\begin{align}
p_i(z+s_1)+p_i(z+s_2) + p_i(z+s_3) &= p_{i-1}(z) + p_i(z) + p_{i+1}(z), \label{eq:pi} \\
p_0(z+s_1)+p_0(z+s_2) + p_0(z+s_3) &= p_0(z) + p_1(z), \label{eq:p0} \\
p_H(z+s_1)+p_H(z+s_2) + p_H(z+s_3) &= \left \{\begin{array}{ll}
p_H(z) + p_{H-1}(z) & \textrm{ if }L\textrm{ is odd}\\ p_{H-1}(z) & \textrm{ if }L\textrm{ is even} \label{eq:pH}
\end{array} \right. ,
\end{align}
hold.
\label{lem:identity_profile}
\end{Lemma}
\begin{proof}
For $z = i e_1 + j e_2 + k e_3 \in \T_L$, let $Pol_z(x)$ be the polynomial of Definition~\ref{def:profile}, that is
\[Pol_z(x) = \frac{(1 - x^{i+1})(1 - x^{j+1})(1- x^{k+1})}{(1-x)^2}. \]
We also extend for any integer $i$ the definition of $p_i(z)$ as the coefficient of $x^i$ in $Pol_z(x)$.

By an inelegant but simple expansion, one can check the identity
\[Pol_{z+s_1}(x)+Pol_{z+s_2}(x)+Pol_{z+s_3}(x) = \left(x + 1 + \frac 1 x\right) Pol_z(x) + x^{L+2}  - \frac 1 x.\] 
Extracting the coefficient of $x^i$ in the above equality for $i \in \{0,\dots,H\}$ straightforwardly gives
 \[p_i(z+s_1)+p_i(z+s_2) + p_i(z+s_3) = p_{i-1}(z) + p_i(z) + p_{i+1}(z),\]
 which proves \eqref{eq:pi}. The equality \eqref{eq:p0} comes from the fact that $p_{-1}(z) = 0$.
 
Concerning $i = H$, we remark that
\[x^{L+1}Pol_z(1/x)= -Pol_z(x),\]
and hence $p_{L+1-j}(z) = -p_j(z)$ for every integer $j$.
In particular, if $L = 2H+1$, then for $j = H + 1$, we have $p_{H+1}(z) = -p_{H+1}(z)$ and so $p_{H+1}(z) = 0$. Equality \eqref{eq:pH} is then obtained by substituting $i=H$ and $p_{i+1} = 0$ in \eqref{eq:pi}.
As for $L = 2H$ even, set $j=H$, and get $p_{H+1}(z) = -p_H(z)$, which implies that only the term $p_{H-1}(z)$ does not disappear in the right-hand side of the equality.
\end{proof}

Thus, Proposition~\ref{prop:motzkin_inductive} is naturally extended to any point of $\T_L$ (not only the ones on the border).

\begin{Theorem}
Let $z$ be any point of $\T_L$ and $(p_0(z),\dots,p_H(z))$ be the profile of $z$. Let us denote $f_n(z)$ the number of forward paths in $\T_L$ starting from $z$. We have
\[f_n(z)= \sum_{i=0}^{H} p_i(z) \Mo{H}{n}{i},\]
where $\Mo H n i$ is the number of Motzkin meanders of length $n$ starting at height $i$ and having an amplitude bounded by $L$.
\label{theo:anywhere}
\end{Theorem}

\begin{proof}
We only do the proof for the odd case, since the even case is very similar. We proceed to an induction on $n$. 

For $n=0$, we have $p_0(z)=1$ since it is the constant term in the polynomial $\frac{(1 - x^{i+1})(1 - x^{j+1})(1- x^{k+1})}{(1-x)^2}$. Moreover, $\Mo H 0 i$ is equal to $0$ if $i > 0$, and $\Mo H 0 0  = 1$. We consistently find $f_0(z) = 1$.

Let us assume that the equality holds for a given $n$ and for every $z' \in \T_L$. We have
\begin{align*}
f_{n+1}(z) &= f_{n}(z + s_1) + f_{n}(z + s_2) + f_{n}(z + s_3) \\
& =  \sum_{i=0}^H \left( p_i(z+s_1)+p_i(z+s_2) + p_i(z+s_3)  \right) \Mo H {n} i & \textrm{by induction,}\\
& =  \sum_{i=1}^{H - 1} \left( p_{i-1}(z)+p_i(z) + p_{i+1}(z)  \right) \Mo H {n} i \\ & + (p_0(z) + p_1(z)) \Mo H {n} 0 + (p_{H-1}(z) + p_H(z)) \Mo H {n} H  & \textrm{by Lemma~\ref{lem:identity_profile}.}
\end{align*}
Collecting terms with respect to $p_i(z)$, we get
\begin{multline*}
f_{n+1}(z) = p_0(z) \left( \Mo H {n} 0 + \Mo H {n} 1 \right) \\ + \sum_{j = 1}^{H-1} p_j(z) \left( \Mo H {n} {j-1} + \Mo H {n} {j}  + \Mo H {n} {j+1} \right) \\ 
 +  p_H(z) \left( \Mo H {n} {H-1} + \Mo H {n} H\right),
\end{multline*}
which reads $f_{n+1}(z) = \sum_{j = 0}^H p_j(H) \Mo H {n+1} j$.
\end{proof}

Let us explain why Proposition~\ref{prop:motzkin_inductive} is a special case of the previous theorem. Given a point of the border $\origin + \ell s_1 = s_1 e_1 + (L-\ell) e_3$  with $\ell \leq H = \lfloor L/2 \rfloor$, the associated polynomial is 
\[\frac{(1-x^{\ell+1})(1-x^{L - \ell + 1})}{1-x} = \left( 1 + x + \dots + x^\ell \right) (1- x^{L-\ell+1}).\]
But since $\ell \leq H$, we have $L - \ell + 1 > H$. So the profile of  $\origin + \ell s_1$ follows the expansion of $1 + x + \dots + x^\ell$. In other words, 
\[p_i(\origin + \ell s_1) = \left\{\begin{array}{cl}
1 & \textrm{ if }i \leq \ell \\
0 & \textrm{ otherwise}
\end{array}
\right. .
\]
We thus recover the formula $f_n(\origin+\ell s_1)=\sum_{i=0}^{\ell} \Mo{H}{n}{i}$.

\subsection{Scaffoldings and new bijections}

In order to illustrate the following definition, we begin this subsection by explaining the idea behind the bijection we are going to present next.

\begin{figure}
\begin{center}
\includegraphics[scale = 1.6]{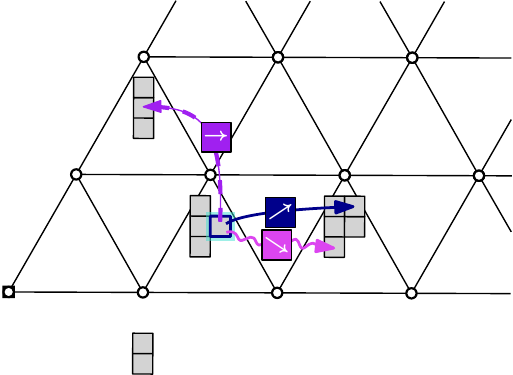}
\end{center}
\caption{A zoom on a scaffolding -- more specifically it depicts the function $s \mapsto \delta_{e_1+ e_2 + 3 e_3}( (1,2),s)$.}
\label{figure:introduction_scaffolding}
\end{figure}

By Theorem~\ref{theo:anywhere}, we know there should be a bijection between the set of triangular paths starting at $z \in T_L$ and the set of triplets $(m,c)$ where $m$ is a Motzkin meander of bounded amplitude and $c$ is a cell in the cell representation of $z$ such that $h(c)$ is the starting height of $m$. 
For the sake of example, let us choose $L = 5$, $z = e_1 + e_2 + 3 e_3$, $c = (f,\ell) = (1,2)$. It corresponds to a specific cell of the profile of $z$, which is highlighted in Figure~\ref{figure:introduction_scaffolding}. 

We now consider a Motzkin path $m$ which we wish to transform into a triangular path starting at $z$, in a recursive manner. This transformation will depend on the cell we have chosen (here $(1,2)$).  At this point there are naturally three  possibilities: $m$ begins by $\nearrow$, by $\rightarrow$, or by $\searrow$. The idea is then to map these three possibilities to three other cells located in the profiles of the neighbors of $z$. The $f$-coordinates of these cells must be respectively $2$, $1$ and $0$. We then use a recursion, which now depends on the new cell, to find the {desired} triangular path. 

Of course there are several choices for these new cells. For example, if $m$ begins by $\nearrow$, we have $3$ choices: there are $2$ cells in the top floor of $z+s_1$, $1$ cell in the top floor of $z+s_2$, and $0$ cell in the top floor of $z+s_3$. Following Figure~\ref{figure:introduction_scaffolding}, we choose the cell $(2,2)$ from the cell representation of $z+s_1$. The triangular path we would like to output will begin by $s_1$ (because the chosen cell is in the profile of $z + \boldsymbol{s_1}$), and the rest will be computed by recursion.

A \textit{scaffolding} is precisely the data which dictates the choice of the new cells for the whole lattice. More precisely, it indicates in which cell we have to go when we consider a specific cell in some profile, and a particular step in $\{\nearrow,\rightarrow,\searrow\}$.

\begin{Definition}[Scaffolding] Let us fix $L$ the size of the triangular lattice, and let $H$ be $\lfloor L/2 \rfloor$.

For a height $f \in \{0,\ldots,H\}$, we say that a step $s\in\{\nearrow,\rightarrow,\searrow\}$ is an \emph{allowed step} from height $f$ if it is a possible step from height $f$ in a Motzkin meander. Precisely, the only restrictions are that $(f,s)$ cannot be equal to $(0,\searrow)$ nor $(H,\nearrow)$, and furthermore, if $L$ is even, $ (f,s)$ cannot be equal to $(H,\rightarrow)$.

For $z\in \T_L$, we define the set 
\[A(z):=\{(c,s)\in \C(z)\times \{\nearrow,\rightarrow,\searrow\} : s \mbox{ is an allowed step from } h(c)\},\]
where $\C(z)$ is  the cell representation of $z$ (see Definition~\ref{def:cell_representation}).
For $i\in\{1,2,3\}$, we also introduce the notation
\[\C_i(z):= \{ (s_i,c) : c \in \C(z)\}.\]
The set $\C_i(z)$ is thus a subset of $\Sf\times \C(z)$, having same cardinality as $\C(z)$, since all the elements of $\C_i(z)$ have the same first coordinate $s_i$.

A \emph{scaffolding} is a collection of functions $(\delta_z)_{z \in \T_L}$, such that for each $z\in \T_L$, the function 
$$\delta_z : A(z) \to \C_1(z+s_1)\cup \C_2(z+s_2) \cup \C_3(z+s_3)$$
is a bijection. Furthermore, for every $(c,s) \in A(z)$ with $(\sigma,c') = \delta_z(c,s)$, we have the restriction
\[ h(c') = \left\{\begin{array}{cl} 
h(c) + 1 & \textrm{if }s=\nearrow \\
h(c)  & \textrm{if }s=\rightarrow \\
h(c) - 1 & \textrm{if }s=\searrow 
 \end{array} \right..
\]
\label{def:scaffolding}
\end{Definition}

An entire scaffolding is shown by Figure~\ref{def:scaffolding}.

\begin{figure}
\begin{center}
\includegraphics[scale = 1.5]{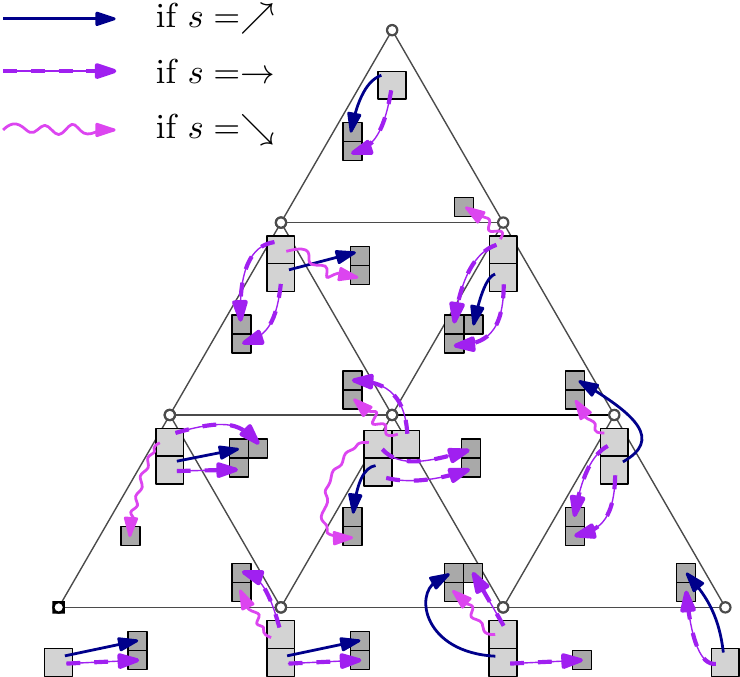}
\end{center}
\caption{A random scaffolding for $\T_3$.
}
\label{figure:example_scaffolding}
\end{figure}

\begin{Proposition} For any $L \geq 0$, there exists a scaffolding.
\end{Proposition}
\begin{proof}
Let us consider any point $z$ of $\T_L$, and let $f'$ be an integer in $\{0,\dots,H\}$.

Consider the sets 
\begin{align*}
\mathcal U_{f'}(z) & := \{ (c,\nearrow) \in A(z) \ : h(c) = f' - 1 \}, 
\\ \mathcal F_{f'}(z) & := \{ (c,\rightarrow) \in A(z) \ :h(c) = f' \}, \\ 
\mathcal D_{f'}(z) & := \{ (c,\searrow) \in A(z) \ : h(c) = f' + 1 \},
\\ \C_{i,f'}(z) & := \{ (s_i,c') \in \C_i(z)\ : \ h(c') = f'\} & \textrm{ for }i \in \{1,2,3\}.
\end{align*}
By Lemma~\ref{lem:identity_profile}, we have
\[ \left| \mathcal U_{f'}(z) \cup \mathcal F_{f'}(z) \cup  \mathcal D_{f'}(z) \right| = \left| \C_{1,f'}(z+s_1) \cup \mathcal \C_{2,f'}(z+s_2) \cup \C_{3,f'}(z+s_3)  \right|. \]
We can then choose any bijection $b_{f'}$ between these two sets and define $\delta_z(c,s)$ for every $(c,s) \in \mathcal U_{f'}(z) \cup \mathcal F_{f'}(z) \cup  \mathcal D_{f'}(z)$ as $b_{f'}(c,s)$.

Doing so for every $f' \in \{0,\dots,H\}$ enables to cover every pair $(c,s) \in A(z)$, and thus successfully define $\delta_z$ on the set of such triplets. 

The required bijectivity of $\delta_z$ is straightforward (because $b_{f'}$ is also bijective).
\end{proof}

\begin{figure}
\begin{center}
\includegraphics[width=0.9\textwidth]{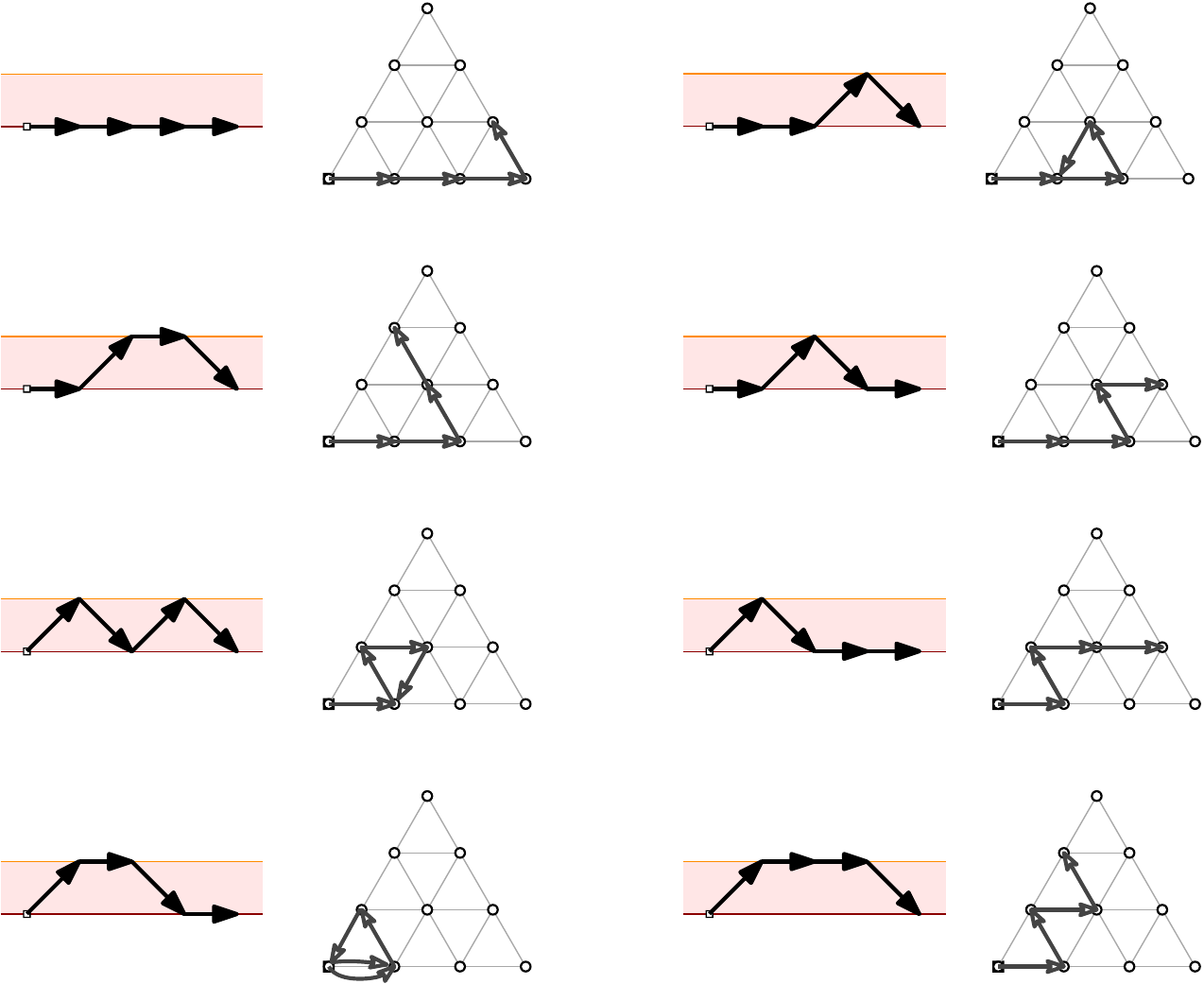}
\end{center}
\caption{The Motzkin paths and the triangular paths of length $3$ in correspondence under Algorithms~\ref{algo:scaffolding1} and~\ref{algo:scaffolding2}, given the scaffolding of Figure~\ref{figure:example_scaffolding}.}
\label{figure:scaffolding-bijection}
\end{figure}

Once we fix a scaffolding for our triangular lattice, one can describe a bijection between triangular paths and Motzkin paths. The bijection is given by Algorithms~\ref{algo:scaffolding1} and~\ref{algo:scaffolding2}.

\begin{algorithm}[caption={Bijection from Motzkin paths to triangular paths, given a scaffolding $(\delta_z)_{z \in \T_L}$ (for \textit{scaffolding}, see Definition~\ref{def:scaffolding}).}, label={algo:scaffolding1}]
metadata: a scaffolding $\delta_z$
input: a Motzkin path m
output: a triangular path p starting at $\origin$
n $\gets$ length of m;
p $\gets$ empty path;
z $\gets \origin$;
c $\gets$ unique cell of height $0$ in the cell representation of z; 
for i from 1 to n 
do ($\sigma$, c) $\gets$ $\delta_{\textrm z}$(c, m[i]);
   add $\sigma$ to the end of p;
   z $\gets$ z + $\sigma$; 
return p;
\end{algorithm}

\begin{algorithm}[caption={Bijection from triangular paths to  Motzkin paths, given a scaffolding $(\delta_z)_{z \in \T_L}$ (for \textit{scaffolding}, see Definition~\ref{def:scaffolding}).}, label={algo:scaffolding2}]
metadata: a scaffolding $\delta_z$
input: a triangular path p starting at $\origin$
output: a Motzkin path m
n $\gets$ length of p;
m $\gets$ empty path;
z $\gets \origin + \sum_{i=1}^n$ p[i];
c $\gets$ unique cell of height $0$ in the cell representation of z; 
for i decreasing from n to 1 
do  (c, s) $\gets$ $\delta_{\textrm z}^{-1}$(p[i], c);
   add s to the beginning of m;
   z $\gets$ z - p[i]; 
return m;
\end{algorithm}

\begin{Theorem}
Let $(\delta_z)_{z \in \T_L}$ be a scaffolding. Algorithms~ \ref{algo:scaffolding1} and~\ref{algo:scaffolding2} give two inverse bijections between the set of Motzkin paths of length $n$ with bounded amplitude $L$ and the set of triangular paths of $\T_L$ of length $n$ starting at $\origin$.
\end{Theorem}

\begin{proof}
At the end of Algorithm~\ref{algo:scaffolding1}, note that the height of the ending cell is $0$, since variable $f$ keeps track of the height of the input Motzkin path (because of the last restriction of Definition~\ref{def:scaffolding}) and a Motzkin path always ends at height $0$. Moreover, because the polynomial $(1-x^{i+1})(1-x^{j+1})(1-x^{k+1})/(1-x)^2$ always has a constant term equal to $1$, by Definition~\ref{def:profile}, we have $p_0(z) = 1$ for every $z \in \T_L$. But $\ell$ is always between $1$ and $p_f(z)$, so at the end of Algorithm~\ref{algo:scaffolding1}, $\ell$ must be $1$.

Thus, the values of $z$ and $c$ are the same at the end of Algorithm~\ref{algo:scaffolding1} and at the beginning of Algorithm~\ref{algo:scaffolding2}. From this point, it is easy to see that the loop of Algorithm~\ref{algo:scaffolding2} reverses what the loop of Algorithm~\ref{algo:scaffolding1} did. Therefore the two algorithms are mutual inverse bijections.
\end{proof}

\begin{Remark}
If we omit the cost of a precalculation (which is the construction of a scaffolding which can be made in $O(L^4)$ time), both algorithms have a linear-time complexity. 

The scaffolding bijection of Subsection \ref{ss:trapeziums} does not require any precalculation (which can be costly if $L$ is large) and it still has a linear-time complexity. 
\end{Remark}

\begin{Remark}If two Motzkin paths $m$ and $m'$ share a common prefix of length $j$, then the two corresponding triangular paths under Algorithm~\ref{algo:scaffolding1} will also share a common prefix of length $j$. The converse is not true.

This property is not shared by the exponential bijection of Figure~\ref{figure:Omega}. This is why this bijection is not a particular case of the scaffolding bijections.
\end{Remark}

\begin{Remark} No scaffolding is necessary if we wish to sample a random forward path under the uniform distribution, given a uniform random Motzkin path of bounded amplitude. 

Indeed, since any scaffolding is suitable to have a bijection, one can pick this scaffolding at random, on the fly. To do so, at each step of the loop in Algorithm~\ref{algo:scaffolding1}, we choose $\delta_z(c,m[i])$ as one of the cells with height $h'$ belonging to $\C(z+s_1)\cup \C(z+s_2) \cup \C(z+s_3)$, where $h'=h(c) + 1$ if $m[i]=\nearrow$, $h'=h(c) $ if $m[i]=\rightarrow$, or $h'=h(c) - 1$ if $m[i]=\searrow$. This choice must be uniform among all cells of height $h'$.
\end{Remark}


\subsection{Two direct bijective proofs of Mortimer and Prellberg's theorem}
\label{ss:bico}

We mention two ways to extend this to a bijection between bounded Motzkin paths with bicolored (black and white) edges and triangular paths (potentially including forward and backward steps), which provides
a direct combinatorial interpretation of Theorem~\ref{theo:mortimerprellberg}. 

The first method is as mentioned at the end of Section~\ref{s:expo}: Starting with a bicolored Motzkin path, use the scaffolding bijection above to send the Motzkin path to a forward path, and map the colors to a direction vector based on the order in which they appear (black $\to F$ and white $\to B$). Then, using the bijection of Theorem \ref{theo:directions}, send the forward path to a path with that direction vector.

For the second method we start by defining a {\em reverse scaffolding} 
\[\overline{\delta_z} : A(z) \to \overline{\C_1}(z+\overline{s_1})\cup \overline{\C_2}(z+\overline{s_2}) \cup \overline{\C_3}(z+\overline{s_3}),\]
where each $\overline{\C_i}(z)$ is defined by
\[
\C_i(z):= \{(\overline{s_i},c) : c\in\C(z)\}.\]
We define $\overline{\delta_z}$ symmetrically to $\delta_{z}$ reflected about the midline of $\T_{L}$ passing though $\origin=x_{3}e_{3}$. To be precise, if $z=x_{1}e_{1}+x_{2}e_{2}+x_{3}e_{3}$, let $z'=x_{2}e_{1}+x_{1}e_{2}+x_{3}e_{3}$ and $\delta_{z'}(a)=(s_{j},c)$. Then  we define $\overline{\delta_{z}}(a):=(\overline{s_{4-j}},c)$. This is possible because the cell representation of $z'$ is necessarily the same as that of $z$. The bijection then runs as follows: starting with a bicolored Motzkin path, we apply the scaffolding $\delta_{z}$ when there is a black step, and we apply the reverse scaffolding $\overline{\delta_z}$ when there is a white step. An advantage of that second version is that it takes linear time to apply.

\subsection{A canonical scaffolding in terms of colored trapeziums}
\label{ss:trapeziums}

In this section we provide an explicit scaffolding which yields a bijection between bounded Motzkin paths and triangular paths which takes linear time to compute (it does not depend on $L$). First we define a new cell representation for $\T_L$.


\begin{Proposition}
For every $z \in \T_L$, the set
\[\C(z):=\left\{(f,\ell)\in\mathbb{Z}^{2}\ | \ \max(0,f-x_{3}) \leq \ell \leq \min(f,x_{1},x_{2},x_{1}+x_{2}-f)\right\}\]
is a cell representation of $z$ (see Definition~\ref{def:cell_representation}).
\label{prop:cell_rep}
\end{Proposition}

\begin{figure}
\begin{center}\hfill
\begin{minipage}{0.4\textwidth}
\includegraphics{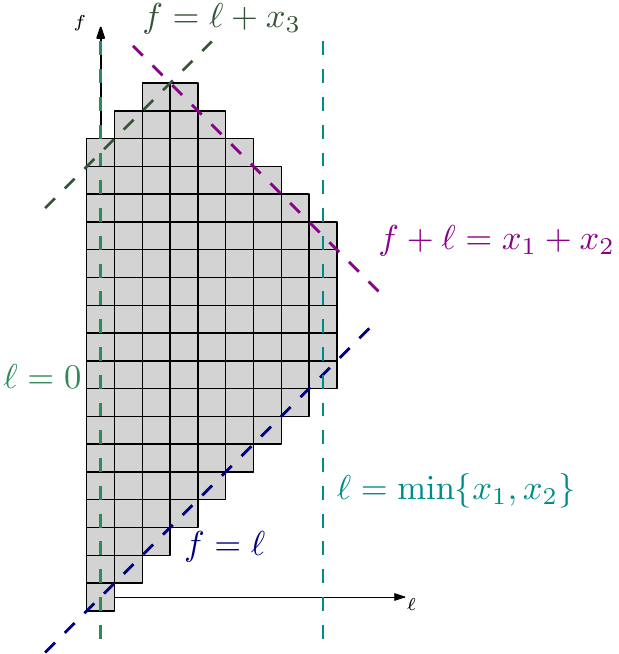}
\end{minipage}\hspace*{\fill}\hspace*{\fill}
\begin{minipage}{0.4\textwidth}
\includegraphics[width = \textwidth]{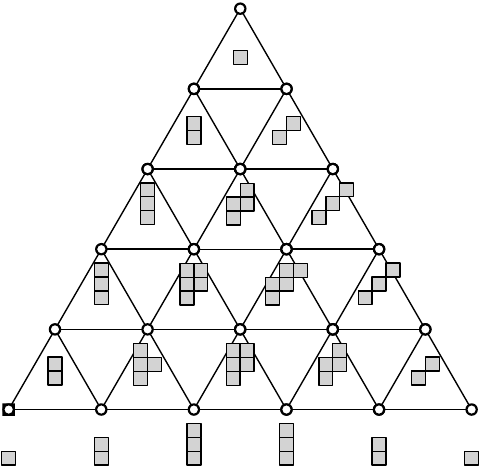}
\end{minipage}
\hspace*{\fill}
\end{center}
\caption{\textit{Left.} The shape of the cell representation from Proposition~\ref{prop:cell_rep} of a point $x_1 e_1 + x_2 e_2 + x_3 e_3$. \textit{Right.} The associated cell representation of $\T_5$.  }
\label{figure:cell_representation}
\end{figure}

\begin{proof}
Set $z=x_{1}e_{1}+x_{2}e_{2}+x_{3}e_{3}$, so $x_{1}+x_{2}+x_{3}=L$. Recall that for \[p_{f}(z)=[y^{f}](1+\cdots+y^{x_{1}})(1+\cdots+y^{x_{2}})(1-y^{x_{3}+1}),\]
for $2f\leq L$. For $x_{3} \geq x_{1}+x_{2}$, an expansion of the two first factors shows that the numbers $p_f(z)$ are 
\[1,2,\ldots, \underbrace{\min(x_{1},x_{2})+1,\min(x_{1},x_{2})+1,\ldots,\min(x_{1},x_{2})+1}_{\textrm{repeated }\max(x_{1},x_{2})-\min(x_{1},x_{2})+1\textrm{ times}},\min(x_{1},x_{2}),\dots,2,1,\]
 for $f=0,1,\ldots,x_{1}+x_{2}$. So, if we simply define $\C(z):=\{(f,\ell) \ |\  0\leq\ell\leq p_{f}(z)-1\}$ with $h((f,\ell))=f$, then $\C(z)$ can alternatively be written as
\[\C(z)=\left\{(f,\ell)\in\mathbb{Z}^{2}\ | \ 0\leq \ell \leq \min(f,x_{1},x_{2},x_{1}+x_{2}-f)\right\}.\]
For $x_{3}<x_{1}+x_{2}$, it suffices to remove from $\C(z)$ any points $(f,\ell)$ for which $(f,\ell-x_{3}-1)$ belongs to $\C(z)$, as this corresponds to multiplying the polynomial by $(1-y^{x_{3}+1})$. This yields the above general formula for $C(z)$.
\end{proof}

Examples of the cell representation of Proposition~\ref{prop:cell_rep} are shown in Figure~\ref{figure:cell_representation}.

\begin{figure}
\begin{center}
\includegraphics[width=0.95\textwidth]{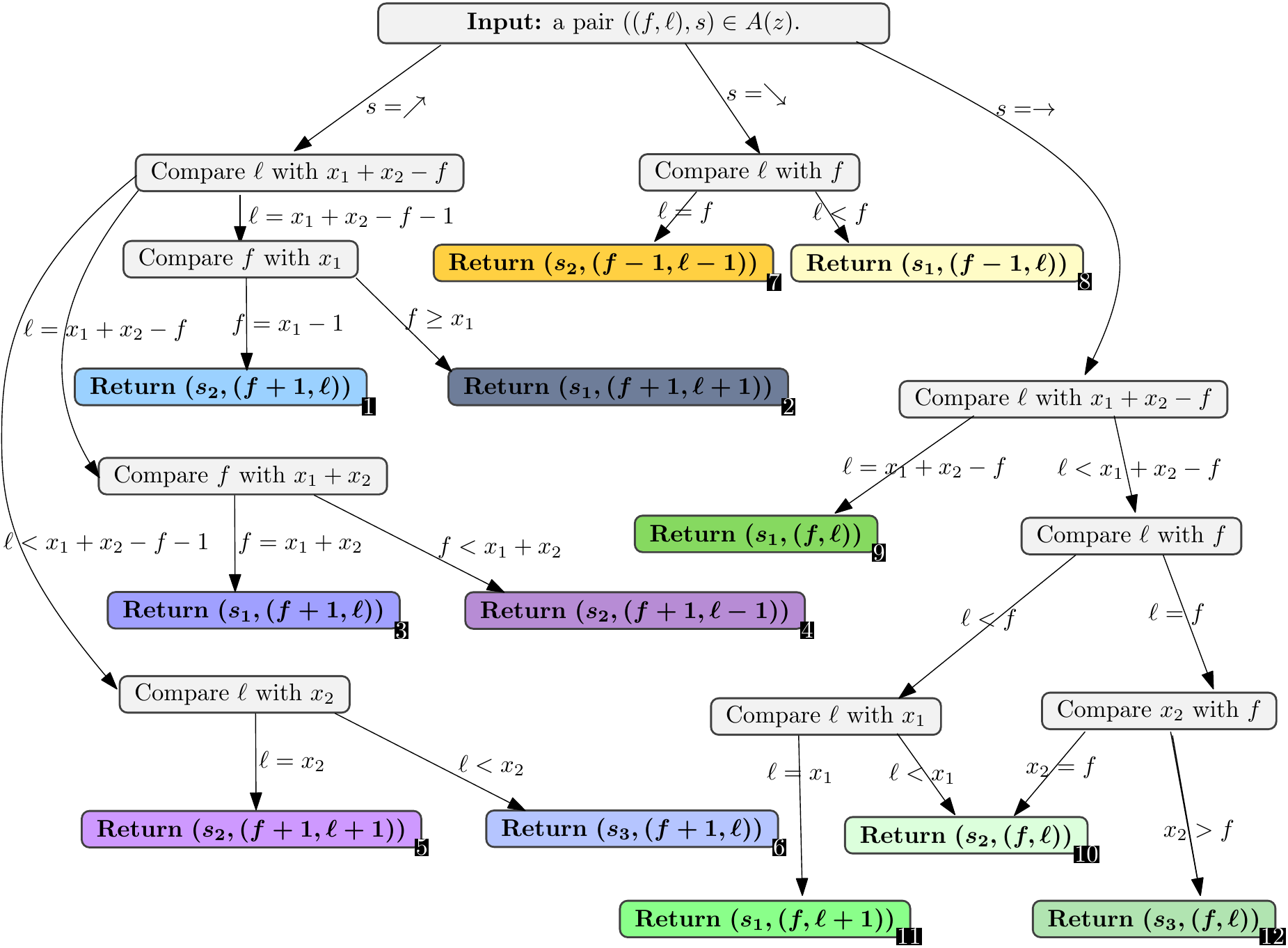}
\end{center}
\caption{A diagram defining the scaffolding $\delta_{z}$.}
\label{figure:trapezium_scaffolding}
\end{figure}

\begin{figure}
\begin{center}
\includegraphics[scale=0.7]{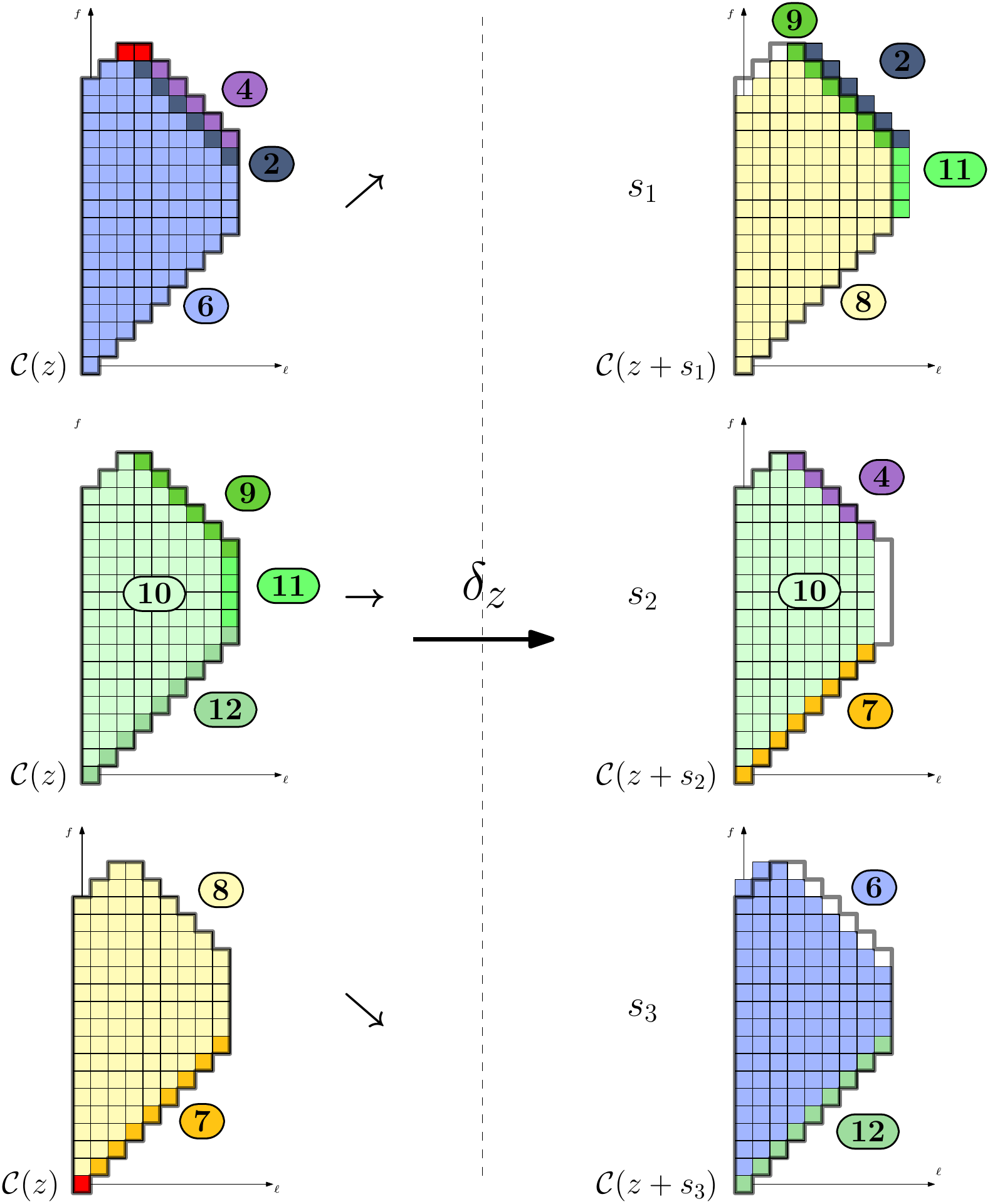}
\end{center}
\caption{A geometric depiction of the bijection $\delta_{z}$ in the case $x_{1}=8$, $x_{2}=13$, $L=37$. On the left, there are three copies of of $\C(z)$ while on the right we have $\C(s_{1}+z)$, $\C(s_{2}+z)$ or $\C(s_{3}+z)$.
Each case in Figure \ref{figure:trapezium_scaffolding} is represented by a colored zone with labels matching the numbers shown in Figure \ref{figure:trapezium_scaffolding}.
Cells for which the given step is not allowed are colored in red.    
The grey polygon on each of the cell representations is the outline of $C(z)$ (equivalently, the pentagon delimited by the lines $\ell = 0$, $f = \ell$,  $\ell = 8$, $f + \ell = 21$ and $f = \ell +  16$).} 
\label{figure:trapeziums_diagram}
\end{figure}

\begin{figure}
\begin{center}
\includegraphics[scale=0.7]{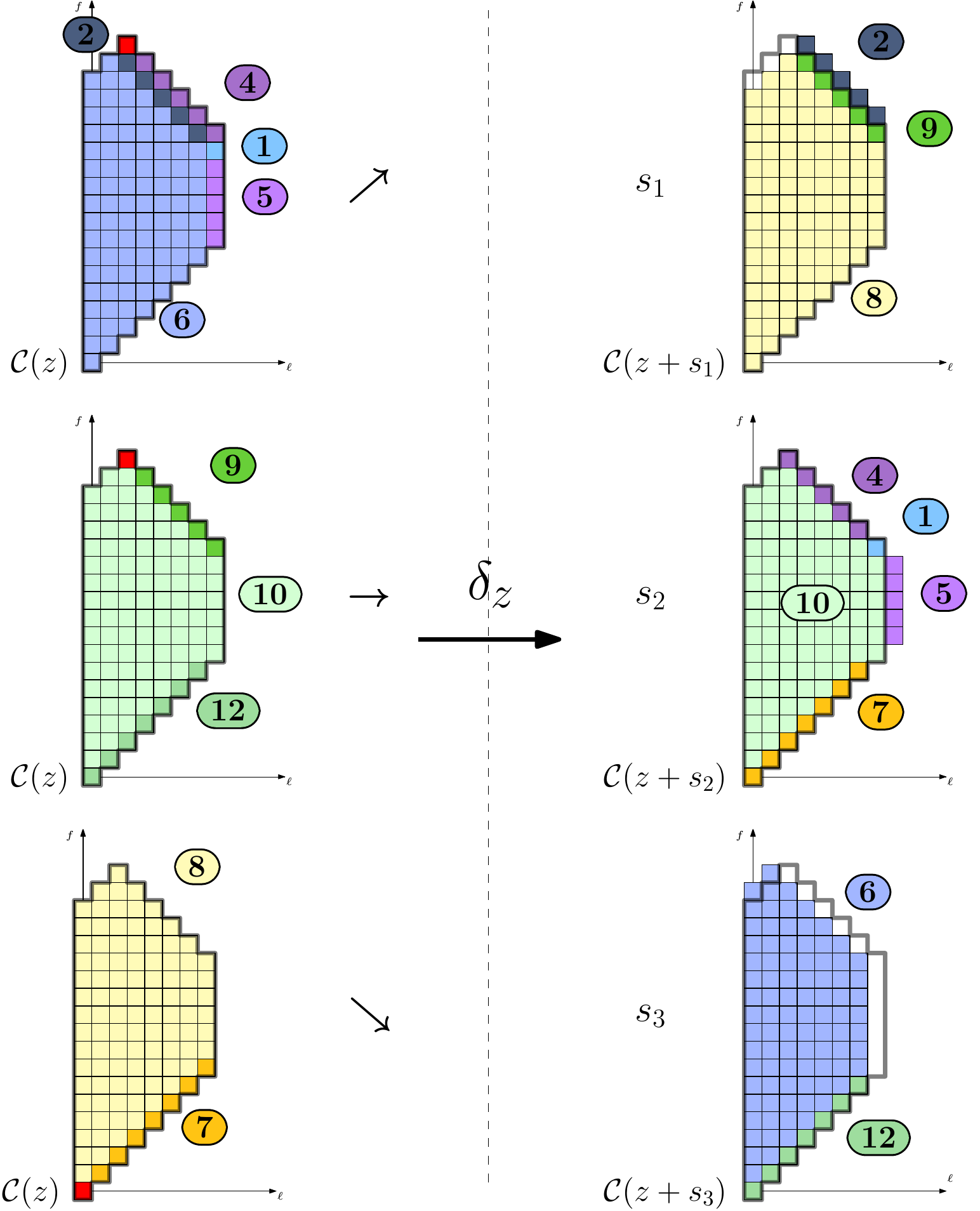}
\end{center}
\caption{The bijection $\delta_{z}$ in the case $x_{1}=13$, $x_{2}=7$, $L=36$. In comparaison with Figure~\ref{figure:trapeziums_diagram}, this decomposition features the case where $L$ is even, but most importantly, the case where $x_1 > x_2 +1$.}
\label{figure:trapeziums_diagram2}
\end{figure}

Finally it remains to define a scaffolding
$$\delta_z : A(z) \to \C_1(z+s_1)\cup \C_2(z+s_2) \cup \C_3(z+s_3),$$
where we recall that $A(z)$ and $\C_i(z)$ are defined by
\begin{align*}A(z)&:=\{(c,s)\in \C(z)\times \{\nearrow,\rightarrow,\searrow\} : s \mbox{ is an allowed step from } h(c)\}.\\
\C_i(z)&:= \{(s_i,c) : c\in\C(z)\}.\end{align*}
We define $\delta_{z}$ by the procedure shown in Figure \ref{figure:trapezium_scaffolding}. Under this procedure there are 12 different cases, shown by the {colored} boxes labeled from $1$ to $12$.

In the following theorem we show that this is indeed a bijection. We give a geometric interpretation of this bijection in two specific cases in Figures \ref{figure:trapeziums_diagram} and \ref{figure:trapeziums_diagram2}. 
\begin{Theorem}
For each $z\in \T_L$, the function $\delta_{z}$ defined by the procedure in Figure \ref{figure:trapezium_scaffolding} is a bijection from $A(z)$ to $\C_1(z+s_1)\cup \C_2(z+s_2) \cup \C_3(z+s_3)$.
\label{theo:trapezium}
\end{Theorem}
\begin{proof}
To see that this is a bijection, is suffices to show that each element of $\C_1(z+s_1)\cup \C_2(z+s_2) \cup \C_3(z+s_3)$ is covered exactly once by $\delta_{z}$. 

First, we claim that $\C_{1}(z+s_{1})$ is covered by cases $2$, $3$, $8$, $9$ and $11$. Note that $\C(z+s_1)=\{(f',\ell')\in\mathbb{Z}^{2}\ |\ \max(0,f'-x_{3}+1) \leq \ell' \leq \min( f', 1+x_{1},x_{2},1+x_{1}+x_{2}-f')\}.$
In particular, the pairs $(f',\ell')\in\C(z+s_{1})$ covered by each of the five cases are those satisfying the following: 
\begin{itemize}
\item Case 2: $\ell'=1+x_{1}+x_{2}-f'\neq 0$.
\item Case 3: $\ell'=0=1+x_{1}+x_{2}-f'$ (this case only occurs if $x_{1}+x_{2}\leq x_{3}$ i.e., $2(x_{1}+x_{2})\leq L$ ).
\item Case 8: $\ell'\leq x_{1}$ and $\ell'<x_{1}+x_{2}-f'$.
\item Case 9: $\ell'\leq x_{1}$ and $\ell'=x_{1}+x_{2}-f'$.
\item Case 11: $\ell'=x_{1}+1\leq x_{1}+x_{2}-f'$ (this case only occurs for $x_{1}<x_{2}$).
\end{itemize}
Next, we show that the set $\C_{2}(z+s_{2})$ is covered by cases $1$, $4$, $5$, $7$ and $10$. We have \[\C(z+s_2)=\{(f',\ell')\in\mathbb{Z}^{2}\ |\ \max(0,f'-x_{3})\leq \ell'\leq \min(f', x_{1}-1,x_{2}+1,x_{1}+x_{2}-f')\}.\]
 In particular, the pairs $(\ell',f')\in\C(z+s_{2})$ covered by each of the five cases are those satisfying the following:
\begin{itemize}
\item Case 1:  $f'=x_{1}$ and $\ell'=x_{2}$ (this case only occurs if $x_{2}\leq x_{1}-1$).
\item Case 4: $\ell'=x_{1}+x_{2}-f'\leq x_{2}-1$.
\item Case 5: $\ell'=x_{2}+1$ (this case only occurs if $x_{2}+1\leq x_{1}-1$).
\item Case 7: $\ell'=f'\leq x_{2}-1$. 
\item Case 10: $\ell'\leq x_{1}+x_{2}-f'-1,x_{2},f'-1$ or $\ell'=f=x_{2}$ (the latter case only occurs for $x_{2}\leq x_{1}-1$).
\end{itemize}
Finally, we show that the set $\C_{3}(z+s_{3})$ is covered by cases $6$ and $12$.  Note that \[\C(z+s_3)=\{(f',\ell')\in\mathbb{Z}^{2} | \max(0,f'-x_{3}-1)\leq \ell'\leq \min(f',x_{1},x_{2}-1,x_{1}+x_{2}-1-f').\] In particular, the pairs $(f',\ell') \in \C(z+s_{3})$ covered by each of the five cases are those satisfying the following:
\begin{itemize}
\item Case 6: $\ell'\leq f'-1$.
\item Case 12: $\ell'=f'$.
\end{itemize}
We thus have dealt with every element of $\C_1(z+s_1)\cup \C_2(z+s_2) \cup \C_3(z+s_3)$.
\end{proof}

Note that the rules in the definition of $\delta_{z}$ only depend on $x_{1}$, $x_{2}$, $f$ and $\ell$, but not $L$. As a consequence, this bijection can be applied to any Motzkin path to yield a path in the $1/6$-plane, and if $L$ is the minimum sidelength of a triangle containing the resulting path then $L$ is the amplitude of the Motzkin path.

\section{Generalization to further dimension}
\label{s:generalization}

This section explains to what extent the results of the previous sections can be generalized. In fact, there is a natural extension of triangular paths to higher dimension (already introduced by \cite{MortimerPrellberg}) for which there still exists a bijective correspondence between forward and backward paths. More surprisingly, we can find in dimension $3$ a new bijection between two families of lattice walks, which is an analogue of the bijection between triangular paths and Motzkin path of bounded amplitude. 

\subsection{What can be extended in any dimension}

\subsubsection{Definition}

For dimension $d$, let $(e_1,e_2,e_3,\ldots,e_{d+1})$ denote the standard basis of $\R^{d+1}$. For some $L\in\N$, we define the subset $\Ss_{d,L}$ of $\N^{d+1}$ as the simplicial section of side length $L$ of the integer lattice:
$$\Ss_{d,L}=\{x_1\,e_1+\cdots+x_{d+1}\,e_{d+1} : x_1, \ldots, x_{d+1}\in\N, x_1+\cdots+x_{d+1}=L\}.$$
We will consider walks in this simplex using \textit{forward} steps $s_{j}=e_{j}-e_{j-1}$ for $1\leq j\leq d+1$ (with the convention that $s_{0}=s_{d+1}$) and \textit{backward} steps $-s_{j}$. Paths of $\Ss_{d,L}$ only using forward steps are again called \textit{forward paths}. The \textit{origin} of $\Ss_{d,L}$, denoted $\origin$, is defined as $L e_{d+1}$. The triangular lattice $\T_L$ can be recovered by setting $d=2$ -- in other words $\T_L = \Ss_{2,L}$.

{As in } the triangle case, forward paths of $\Ss_{d,L}$ starting from $\origin$ form a subfamily of Standard Young Tableaux. Precisely, they are in bijection with standard Young tableaux with $d$ rows or less with an extra restriction:
  for $i>L$, if there is a cell with label $\ell$ at position $i$ in the top row of the Young tableau, then there is a cell at position $i-L$ in the bottom row of the Young tableau with a label less than $\ell$. The enumeration of standard Young tableaux with a bounded number of rows is the object of a very active research -- see \cite{mishna} for a survey.

\subsubsection{Equinumeracy of forward and backward paths}

Defining direction vector as in Definition \ref{def:direction_vector}, the equivalent of Theorem \ref{theo:directions} still holds:
\begin{Theorem}
Given two sequences $W$ and $W'$ of $\{\Sf,\Sb\}^n$, the set of paths in $\Ss_{d,L}$ of direction vector $W$ are in bijection with the set of pyramid paths of direction vector $W'$. 
\label{theo:generic_forward}
\end{Theorem}
We can use the same proof almost \textit{verbatim}. In fact, the bijection uses swap flips, defined exactly as in Definition \ref{def:flips}:
\begin{align*}
(s_j,\overline{s_k}) \lra (\overline{s_k}, s_j)& \quad \mbox{if} \; j\not=k, \\ (s_k,\overline{s_k}) \lra (\overline{s_{k-1}},s_{k-1}) & \quad  \mbox{otherwise.}
\end{align*}
where, by convention, $s_0 = s_{d+1}$.

\subsection{Dimension 3}

It turns out that forward paths in dimension $3$ are {equinumerous} with another family of paths, as in the two dimensional case. We {will show this inductively, then give} a bijection {analogous} to those in Section~\ref{s:scaffolding}.

In dimension $3$, the set
$$\Ss_{3,L}=\{x_1\,e_1+x_2\,e_2+x_3\,e_3+x_4\,e_4: x_1, x_2, x_3, x_4\in\N, x_1+x_2+x_3+x_4=L\}$$
is a pyramidal lattice, as shown by Figure~\ref{figure:chef_doeuvre2} (left). We denote by $\Sf$ the set of forward steps, i.e., $\Sf=\{e_1-e_4,e_2-e_1,e_3-e_2,e_4-e_3\}$, and we denote by $\Sb$ the set of backward steps, i.e., $\Sb=-\Sf$. A \textit{pyramidal walk} is a walk in $\Ss_{3,L}$ using steps in $\Sf \cup \Sb$.

\begin{figure}
\begin{center}
\begin{minipage}{.45\textwidth}
\includegraphics[width = \textwidth]{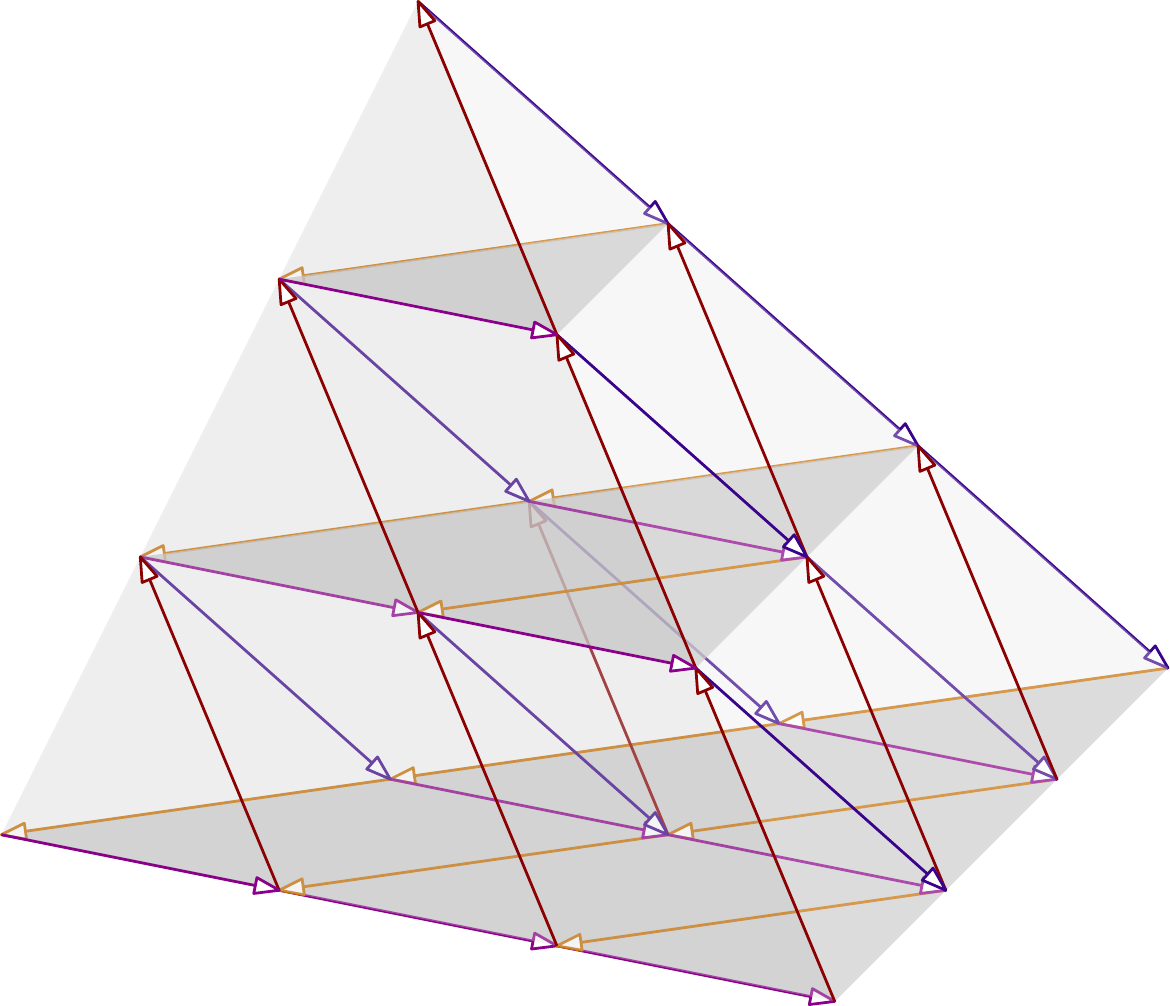}
\end{minipage}
\hspace*{\fill}
\begin{minipage}{.45\textwidth}
\includegraphics[width = \textwidth]{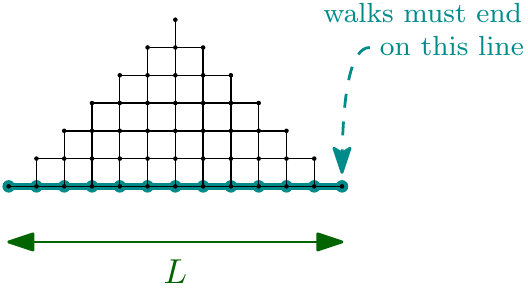}
\end{minipage}
\end{center}
\caption{\textit{Left.} The Pyramid $\Ss_{3,3}$. \textit{Right.} The waffle $W_{12}$. }
\label{figure:chef_doeuvre2}
\end{figure}

By reducing the dimension of the recurrence using the bijection between forward and backward paths, we find a family of paths in bijection with pyramidal walks:

\begin{Theorem}\label{theo:waffle_to_pyramid}
Define the {\em waffle} $W_{L}$ of size $L$ by
\[W_{L}=\{(i,j)\in\mathbb{N} : j\leq i\leq L-j\}\]
(see Figure~\ref{figure:chef_doeuvre2} (right) for a picture).
For $(i,j)\in W_{L}$, the number $w_{n,i,j}$ of square lattice walks in $W_{L}$, starting at $(i,j)$ and ending on the $y$-axis is given by
\[w_{n,i,j}=p_{n,i,j}-p_{n,i-1,j-1},\]
where $p_{n,i,j}$ is the number of forward (or equally backward) pyramid paths of length $n$ starting at the point $(i-j) e_{1}+je_{2}+(L-i)e_{4}$.
\end{Theorem}

\begin{proof}
We prove this using an inductive approach. We define $q_{n,i,j}$ to be the number of such paths starting at the point $(i-j) e_{1}+je_{2}+e_{3}+(L-i-1)e_{4}$ (this is $0$ if the starting point is outside the region).

Considering the first step in a forwards path of length $n+1$ starting at $(i-j) e_{1}+je_{2}+(L-i)e_{4}$ yields the following equation for $n,i,j\geq 0$ satisfying $i\leq j\leq L$:
\[p_{n+1,i,j}=p_{n,i+1,j}+p_{n,i,j+1}+q_{n,i-1,j-1}.\]
Using the same method for backward paths yields
\[p_{n+1,i,j}=p_{n,i-1,j}+p_{n,i,j-1}+q_{n,i,j}.\]
Canceling the $q$ terms, we obtain the following equation as long as $1\leq j\leq i\leq L$:
\[p_{n+1,i,j}-p_{n+1,i-1,j-1}=p_{n,i+1,j}+p_{n,i,j+1}-p_{n,i-2,j-1}-p_{n,i-1,j-2}.\]
Finally, writing $w_{n,i,j}:=p_{n,i,j}-p_{n,i-1,j-1}$, we have the following recurrence for $w$:
\[w_{n+1,i,j}=w_{n,i+1,j}+w_{n,i,j-1}+w_{n,i,j+1}+w_{n,i-1,j},\]
which has only positive coefficients. By analysing this equation on the boundary, we deduce that it holds for $0\leq j\leq i\leq L+1$, if we define $w_{n,i,j}=0$ for $i,j$ outside this region. Finally the initial condition for $w_{0,i,j}$ follows from $p_{0,i,j}=1$ for $0\leq j\leq i\leq L$:
\begin{align*}
		w_{0,i,j} &= 0, & \text{ for } & 1\leq j\leq i\leq L, \\
		w_{0,i,0} &= 1, & \text{ for } & 0\leq i\leq L,\\
		w_{0,L+1,j} &= -1, & \text{ for } & 1\leq j\leq L+1, \\
		w_{0,L+1,0} &= 0.&&
	\end{align*}
	These initial conditions along with the recurrence uniquely define the terms $w_{n,i,j}$.
Now, by symmetry, $w_{n,i,j}=-w_{n,L+1-j,L+1-i}$, and in particular, $w_{n,i,L+1-i}=0$, so we only need to consider the region $i+j\leq L$. Within this region, all terms are positive, so $w_{n,i,j}$ can be understood combinatorially. The combinatorial interpretation of the recurrence is precisely the statement of the theorem: $w_{n,i,j}$ is the number of square lattice walks starting at $(i,j)$ and ending on the $y$-axis, which are confined to the region 
$W_{L}=\{(i,j)\in\mathbb{N}:i\leq j\leq L-i\}$.
\end{proof}

In particular, $p_{n,0,0}=w_{n,0,0}$. 

\begin{Remark}
If we apply the transformation $(x,y) \mapsto (x-y,y)$ to waffle walks, we remark that pyramidal walks starting at $\origin$ are in bijection with \emph{Gouyou-Beauchamps walks}, i.e. walks with North-West, West, East, South-East steps, going from $(0,0)$ to a point on the $x$-axis and confined in the part of the positive quarter of plane below the line $x + 2y = L$. This is consistent with the fact that standard Young tableaux with $4$ rows or less are in bijection with Gouyou-Beauchamps walks returning to the $x$-axis confined in the quarter of plane~\cite{gouyouBeauchamps}.
\end{Remark}

More generally, the following proposition relates the enumeration of pyramid walks starting at any point to waffle walks.

\begin{Proposition}
\label{prop:waffle_to_pyramid} 
The number $p_{n}(z)$ of length $n$ pyramid walks starting at a point $z=x_{1}e_{1}+x_{2}e_{2}+x_{3}e_{3}+x_{4}e_{4}$ is equal to the number of length $n$ waffle walks starting at a point in the set $W(z)$, defined by
\[W(z):=\{(x_{1}+x_{3}+p-q,p+q):p,q\in\mathbb{N},~p\leq\min(x_{2},x_{4}),~q\leq\min(x_{1},x_{3})\}.\]
\end{Proposition}

Now, we will give a bijective proof of this.  The proof is via a scaffolding, analogous to Definition \ref{def:scaffolding}. Again, before we define scaffolding we define the profile of a point.

\begin{Definition}[Profile]
For a point $z=x_{1}e_{1}+x_{2}e_{2}+x_{3}e_{3}+x_{4}e_{4}$, we define the {\em profile} $\C(z)$ of $z$ by \[\C(z):=\{(p,q)\in\mathbb{N}^{2}:~p\leq\min(x_{2},x_{4}),~q\leq\min(x_{1},x_{3})\}.\] We have a natural bijection $h_{z}:\C(z)\to W(z)$ defined by $h_{z}(p,q):=(x_{1}+x_{3}+p-q,p+q)$.
\end{Definition}

\begin{figure}
\begin{center}
\includegraphics[width=0.8\textwidth]{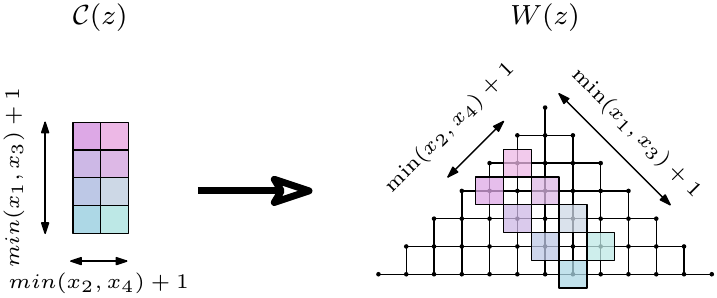}
\end{center}
\caption{The sets $\mathcal C(z)$ and $W(z)$ for $z = 4 e_1 + e_2 + 3 e_3 + 4 e_4$.}
\label{figure:natural_bijection}
\end{figure}

For $z\in \Ss_{3,L}$, we define the set 
$$A(z):=\{(c,s)\in \C(z)\times \{\uparrow,\rightarrow,\downarrow,\leftarrow\} : s \mbox{ is an allowed step from } h_{z}(c)\}.$$
For $i\in\{1,2,3,4\}$, we also introduce the notation
$$\C_i(z):= \{(s_i,c) : c\in\C(z)\}.$$
The set $\C_i(z)$ is thus a subset of $\Sf\times \C(z)$, having same cardinality as $\C(z)$, since all the elements of $\C_i(z)$ have the same first coordinate $s_i$.

\begin{Definition}[Scaffolding] Let us fix the size $L$ of the pyramid.
A \emph{scaffolding} is a collection of functions $(\delta_z)_{z \in \Ss_{3,L}}$, such that for each $z\in \Ss_{3,L}$, the function 
$$\delta_z : A(z) \to \C_1(z+s_1)\cup \C_2(z+s_2) \cup \C_3(z+s_3) \cup \C_4(z+s_4)$$
is a bijection and whenever $\delta_z(c,s)=(s_{j},c_{j})$, we have $h_{z}(c)+s=h_{z+s_{j}}(c_{j})$.
\label{def:scaffolding3d}
\end{Definition}

\begin{figure}
\begin{center}
\includegraphics[width=0.9\textwidth]{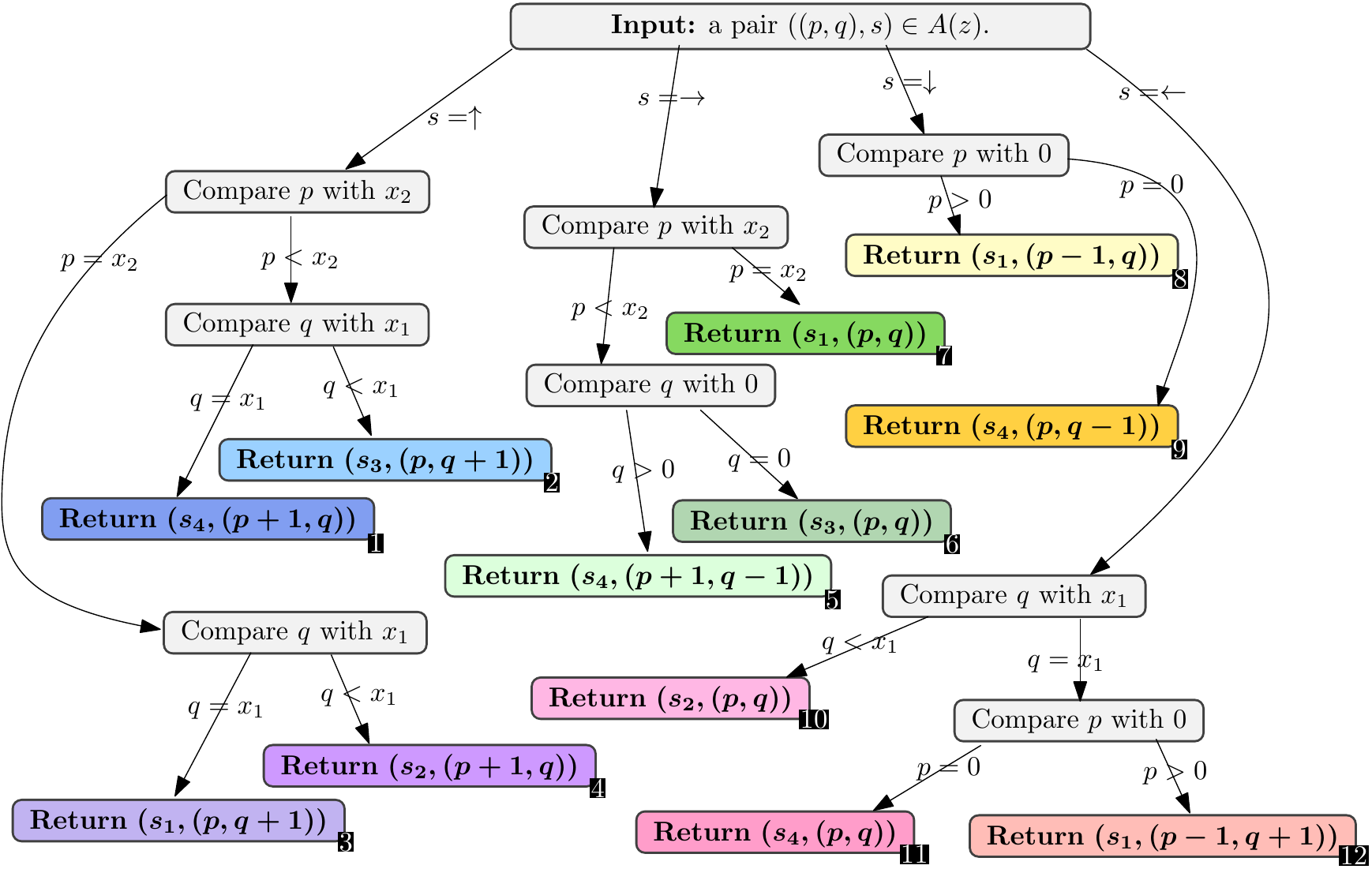}
\end{center}
\caption{A diagram defining the scaffolding $\delta_{z}$.}
\label{diamond_scaffolding}
\end{figure}

Figure~\ref{diamond_scaffolding} shows an example of the sets $W(z)$ and $\C(z)$.

\begin{figure}
\begin{center}
\includegraphics[width=0.9\textwidth]{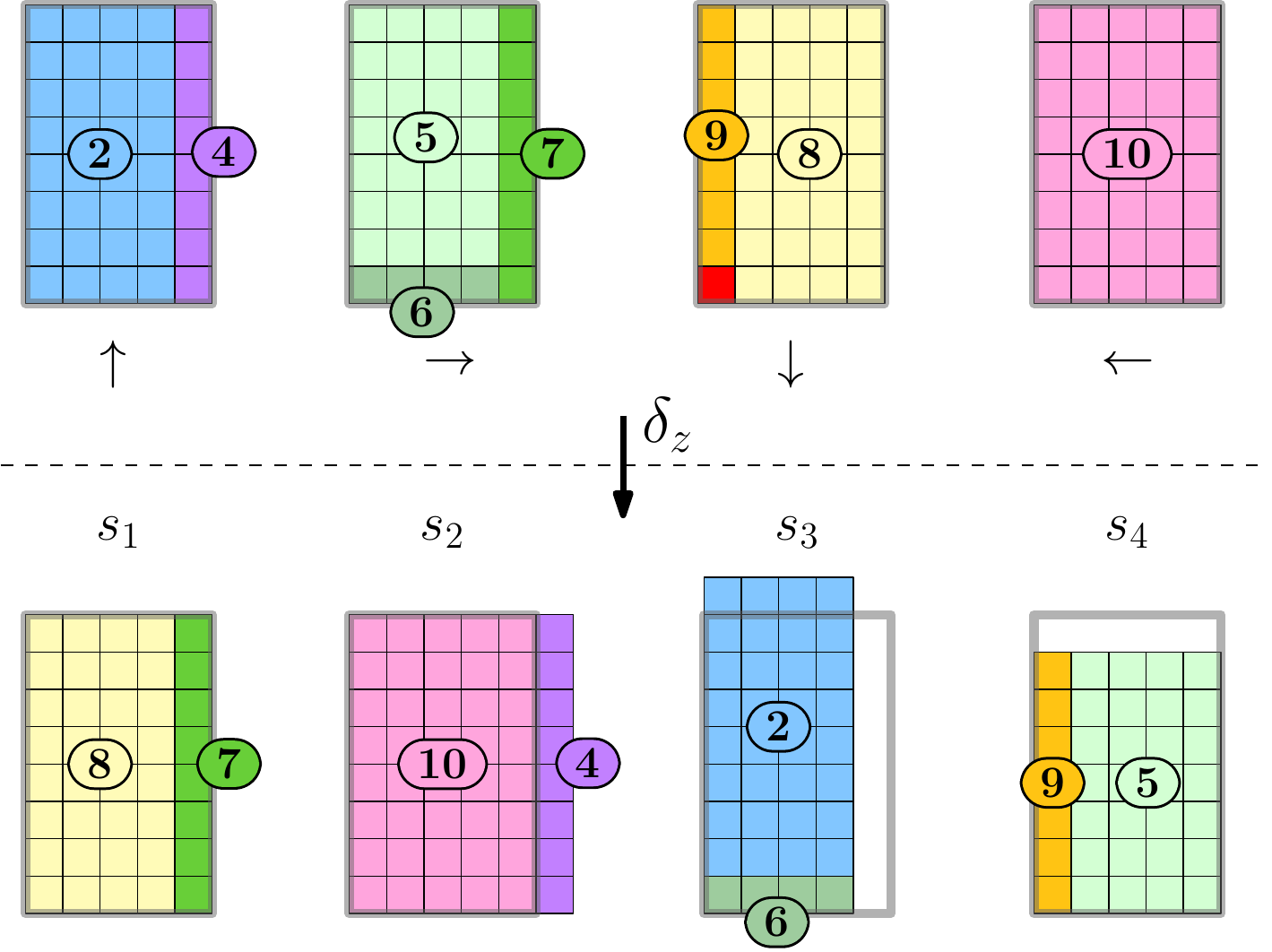}
\end{center}
\caption{A geometric representation of $\delta_z$ for $x_1 = 8$, $x_2 = 4$, $x_3 = 6$ and $x_4 = 7$. The numbers of the colored zones match with cases of the diagram of Figure~\ref{diamond_scaffolding}.}
\label{scaffolding_waffle}
\end{figure}

An explicit scaffolding $\delta_{z}$ is given in Figure \ref{diamond_scaffolding}. The proof of the bijectivity of $\delta_z$ is omitted (because of its tediousness --- it is a case-by-case proof, similar to the one of Theorem~\ref{theo:trapezium}), but some particular configuration is illustrated by Figure~\ref{scaffolding_waffle}. 

Given such a scaffolding, a bijection for each point $z_{c}\in \Ss_{3,L}$ from the set of waffle walks starting at a point in the set $W(z_{c})$ to the set of pyramid walks starting at $z_{c}$ is given by Algorithm~\ref{algo:scaffolding3d}.

\begin{algorithm}[caption={Bijection from waffle paths to pyramid paths, given a scaffolding $(\delta_z)_{z \in \Ss_{3,L}}$ (for \textit{scaffolding}, see Definition~\ref{def:scaffolding3d}).}, label={algo:scaffolding3d}]
metadata: a scaffolding $\delta_{z}$
input: A point $(p_{c},q_{c})\in \C(z_{c})$, a waffle path w starting at $h_{z_{c}}(p_{c},q_{c})$
output: a pyramid path y starting at $z_{c}$.
n $\gets$ length of $w$;
y $\gets$ empty path;
z $\gets z_{c}$;
p $\gets p_{c}$;
q $\gets q_{c}$; 
for i from 1 to n 
do ($\sigma$, p, q) $\gets$ $\delta_{\textrm z}$(f, q, w[i]);
   add $\sigma$ to the end of y;
   z $\gets$ z + $\sigma$; 
return y;
\end{algorithm}

 In the following corollary {of} Theorem \ref{theo:waffle_to_pyramid}, we enumerate pyramidal walks starting at~$\origin$ using the relation $p_{n,0,0}=w_{n,0,0}$, which relates their enumeration to that of waffle walks.
This partially answers another open question of Mortimer and Prellberg~\cite[Section 4.1]{MortimerPrellberg}.

\begin{Corollary}The generating function
\[P(t)=\sum_{t=0}^{\infty}p_{n,0,0}t^{n}\]
for pyramid walks starting in a corner is given by
\[P(t)=\frac{1}{(L+4)^2}\sum_{\substack{1\leq j<k\leq L+3\\2\nmid j,k}}^{L+4}\frac{(\alpha^{k}+\alpha^{-k}-\alpha^{j}-\alpha^{-j})^2(2+\alpha^{j}+\alpha^{-j})(2+\alpha^{-k}+\alpha^{k})}{1-(\alpha^{j}+\alpha^{-j}+\alpha^{k}+\alpha^{-k})t},\]
where $\alpha=e^{\frac{i\pi}{L+4}}$.
\label{cor:enum_pyramid_walks}
\end{Corollary}
\begin{proof}
To prove this, we relate walks confined to the waffle to unconfined walks using the reflection principle~\cite{gessel1992random}, which is possible because the waffle $W_L$ forms a Weyl chamber of some reflection group.

Let $(x,y)$ be a point inside the waffle, let $\Omega$ be the set of unconstrained square lattice walks starting at $(x,y)$ and let $\Omega'$ be the set of walks in the waffle starting at $(x,y)$. Let $\ell_{1}$, $\ell_{2}$ and $\ell_{3}$ be the lines just outside the boundary of $W_{L}$, defined by $y=-1$, $y-x=-1$ and $x+y=L+1$ respectively. We consider the involution $f:\Omega\setminus \Omega'\to\Omega\setminus \Omega'$ defined by reflecting the section of the walk after its first intersection with one of the lines $\ell_{1}$, $\ell_{2}$ and $\ell_{3}$ in that line.

Now, define
\begin{align*}T_{L}&:=((2L+8)\mathbb{Z})\times((2L+8)\mathbb{Z})\cup (L+4+(2L+8)\mathbb{Z})\times(L+4+(2L+8)\mathbb{Z})\\
A_{L}&:=T_{L}\cup\left((-1,-3)+T_{L}\right)\cup\left((-4,-2)+T_{L}\right)\cup\left((-3,1)+T_{L}\right)\\
B_{L}&:=\left((-1,1)+T_{L}\right)\cup\left((0,-2)+T_{L}\right)\cup\left((-3,-3)+T_{L}\right)\cup\left((-4,0)+T_{L}\right).
\end{align*}
Then the involution $f$ sends walks in $\Omega\setminus \Omega'$ ending at a point in $A_{L}$ to walks ending at a point in $B_{L}$ and vice-versa. The only walks in $\Omega'$ ending at a point in $A_{L}$ (or $B_{L}$) are those ending at $(0,0)$. Hence the number of waffle walks of a given length from $(x,y)$ to $(0,0)$ is equal to the number of (uncontrained) walks of the same length from $(x,y)$ to a point in $A_{L}$ minus the number of such walks from $(x,y)$ to a point in $B_{L}$. By shifting the starting point, this is the number of walks from a point in $\{(x,y),(x+1,y+3),(x+4,y+2),(x+3,y-1)\}$ to a point in $T_{L}$ minus the number of walks from a point in $\{(x+1,y-1),(x,y+2),(x+3,y+3),(x+4,y)\}$ to a point in $T_{L}$. These numbers can easily be computed using the generating function for unconstrained walks, and doing so yields the formula in the statement of the theorem. As an example, we show how to compute the generating function for walks from $(x,y)$ to a point in $T_{L}$ counted by length.

Let $F(t,a,b)$ be the generating function for walks starting at $(x,y)$ with walks of length $n$ ending at $(x_{1},y_{1})$ contributing $a^{x_{1}}b^{y_{1}}t^{n}$. We want to sum the coefficients where the powers $x_{1}$ and $y_{1}$ of $a$ and $b$ are both multiples of $2L+8$ or both $L+4$ more than multiples of $2L+8$. For those where both $x_{1}$ and $y_{1}$ are multiples of $2L+8$, This is achieved by setting $\alpha=e^{\frac{i\pi}{L+4}}$, and writing the sum
\[\frac{1}{(2L+8)^2}\sum_{1\leq j,k\leq 2L+7}F(t,\alpha^{j},\alpha^{k}),\]
as the contribution to this sum from a monomial $a^{x_{1}}b^{y_{1}}t^{n}$ is
\[t^n\left(\frac{1}{2L+8}\sum_{1\leq j\leq 2L+7}\alpha^{x_{1}j}\right)\left(\frac{1}{2L+8}\sum_{1\leq k\leq 2L+7}\alpha^{y_{1}k}\right),\]
which is $0$ unless $x_{1}$ and $y_{1}$ are both multiples of $2L+8$, in which case it is $t^n$. Similarly, the generating function for the cases where $x_{1}-L-4$ and $y_{1}-L-4$ are multiples of $2L+8$ is
\[\frac{1}{(2L+8)^2}\sum_{1\leq j,k\leq 2L+7}(-1)^{j+k}F(t,\alpha^{j},\alpha^{k}).\]
Similarly, one can write expressions for the generating function of walks from any given point to a point in $T_{L}$. Adding and subtracting these as appropriate yields the desired result.
\end{proof}


\section{Conclusion}

To sum up, we have found several bijections between forward triangular walks and Motzkin path with bounded amplitude, answering thus  Mortimer and Prellberg's open question~\cite{MortimerPrellberg}.

There were some interesting consequences from this discovery. First, by looking for a bijection, we discovered an unexpected symmetry property between forward and backward paths (Theorem~\ref{theo:directions}). Second, we refined Mortimer and Prellberg's results by considering triangular walks starting not only at {the} origin, but at any point in the triangle (Theorem~\ref{theo:anywhere}). Finally, by mimicking the proof of the first sections, we managed to extend some of our results to larger dimensions. In particular, we discovered a new bijective correspondence in dimension 3 (Theorem~\ref{theo:waffle_to_pyramid}), enabling in the process to find an expression for the generating function of pyramid walks (Corollary~\ref{cor:enum_pyramid_walks}), which was also an open question in Mortimer and Prellberg's paper.

However, we still do not know if there exists a bijection between triangular walks in dimension $d \geq 4$ and some {class} of walks in dimension $d-1$. It seems like our two- and three-dimensional argument (more precisely, the one in the proofs of Proposition~\ref{prop:motzkin_inductive} and Theorem~\ref{theo:waffle_to_pyramid}) does not work anymore. We leave the question of Mortimer and Prellberg about the enumeration of triangular walks in higher dimension as an open question.

There is another conjecture from a different paper that may relate to this current work: the three authors of~\cite{bousquetmelouFusyRaschel} conjecture that there exists a length-preserving involution on
double-tandem walks that exchanges 
$x_{start}-x_{min}$
 and 
 $y_{end}-y_{min}$, 
 while preserving $y_{start}-y_{min}$ and $x_{end}-x_{min}$ (point $(x_{start},y_{start})$ denotes the starting point, and $x_{min}$ and $y_{min}$ are respectively the minimal x- and y-coordinates during the walk). It may be interesting to see if techniques of Section~\ref{s:symmetry} facilitate the discovery of this involution.

Finally, this paper shows two examples of bijections where there is a trade-off between domain and endpoint constraints:
\begin{itemize}
\item The one between triangular paths and Motzkin paths transform two-dimensional walks with no constraint on the endpoint into one-dimensional walks which must finish at the origin;
\item the one between pyramid paths and waffle walks transform three-dimensional walks with no constraint on the endpoint into two-dimensional walks which must end on one of the axis.
\end{itemize}
This is somehow reminiscent of \cite{Eliz15,courtiel2018bijections}. We wonder whether there are some other examples of this phenomenon, or even a generic framework for such bijections.

\bibliographystyle{plain}
\bibliography{mortimer-prellberg}

\end{document}